\documentclass[10pt]{amsart}
\usepackage[utf8x]{inputenc}
\usepackage{verbatim}
\usepackage{amsmath}
\newtheorem{theorem}{Theorem}[section]
\newtheorem{proposition}{Proposition}[section]
\newtheorem{lemma}{Lemma}[section]

\theoremstyle{definition}
\newtheorem*{definition}{Definition}{}

\theoremstyle{remark} 
\newtheorem{remark}{Remark}[section]

\newcommand{\avs}{\langle s \rangle}
\newcommand{\htw}{\mathfrak{h}_2}
\newcommand{\hon}{\mathfrak{h}_1}
\newcommand{\h}{\mathfrak{h}}

\newcommand{\pbp}{\partial \bar{\partial}}

\numberwithin{equation}{section}
\begin{document}
\title[A vector bundle version of the Monge-Amp\`ere equation]{A vector bundle version of the Monge-Amp\`ere equation}
\author{Vamsi Pritham Pingali}
\address{Department of Mathematics, Indian Institute of Science, Bangalore, India - 560012}
\email{vamsipingali@iisc.ac.in}
\begin{abstract} 
We introduce a vector bundle version of the complex Monge-Amp\`ere equation motivated by a desire to study stability conditions involving higher Chern forms. We then restrict ourselves to complex surfaces, provide a moment map interpretation of it, and define a positivity condition (MA-positivity) which is necessary for the infinite-dimensional symplectic form to be K\"ahler.
On rank-2 bundles on compact complex surfaces, we prove two consequences of the existence of a ``positively curved" solution to this equation - Stability (involving the second Chern character) and a  Kobayashi-L\"ubke-Bogomolov-Miyaoka-Yau type inequality. Finally, we prove a Kobayashi-Hitchin correspondence for a dimensional reduction of the aforementioned equation.
\end{abstract}
\maketitle
\section{Introduction}\label{Introsec}
\indent The Kobayashi-L\"ubke-Bogomolov-Miyaoka-Yau (KLBMY) inequality for Mumford stable bundles on K\"ahler manifolds (and its cousin the Bogomolov-Miyaoka-Yau inequality for anti-Fano K\"ahler manifolds)  $$(r-1)c_1^2 [\omega]^{n-2}-2rc_2 [\omega]^{n-2} \leq 0$$
has many complex-geometric applications (see \cite{yaupnas, koblub} for instance). It is natural to ask whether we can produce similar inequalities for higher Chern classes conditioned on more complicated stability conditions. Alternatively, one can attempt to discover Partial Differential Equations whose solvability implies such inequalities. As far as we know, the only result in this direction is due to Collins-Xie-Yau \cite{collinsxieyau}. 
\begin{theorem}[Collins-Xie-Yau]
Let $(M,\omega)$ be a compact K\"ahler 3-manifold and $L$ is a holomorphic line bundle over it. Consider the deformed Hermitian-Yang-Mills (dHYM) equation for a metric $h$ on $L$ whose curvature is $F$ 
\begin{gather}
\mathrm{Arg}\left ( \frac{(\omega-F)^3}{\omega^3} \right ) = \hat{\theta},
\label{dhym}
\end{gather}
where $\hat{\theta}\in (\frac{\pi}{2},\frac{3\pi}{2})$ is a constant. Suppose there exists a solution to \ref{dhym}. Then
\begin{gather}
\displaystyle \int_M \omega^3 \int_M \mathrm{ch}_3(L) \leq 3 \int_{M} \mathrm{ch}_2 (L) \wedge \omega  \int_M \mathrm{ch}_1(L)\wedge \omega^2.
\label{cxyineq}
\end{gather}
\label{linebundlekl}
\end{theorem} 
 In two dimensions the dHYM equation can be written as a Monge-Amp\`ere equation \cite{jacobyau}. In higher dimensions one can get existence results by either treating it directly \cite{collinsjacobyau}, or in some special cases, by rewriting it as a generalised Monge-Amp\`ere equation \cite{pindhym}. \\ 
\indent To extend Theorem \ref{linebundlekl} to vector bundles, one approach would be to consider a naive generalisation of the dHYM equation by replacing the curvature form with the curvature endomorphism. Since the simplest of this flavour of fully nonlinear PDE is the usual complex Monge-Amp\`ere equation, we propose to study the following vector bundle Monge-Amp\`ere (vbMA) equation for a metric $h$ on a holomorphic vector bundle $E$ over a compact complex $n$-dimensional manifold $M$ :     
\begin{gather}
\left ( \frac{i\Theta_h}{2\pi} \right )^n  =\eta Id
\label{vectorbundleMAeqn}
\end{gather}
where $\Theta_h$ is the curvature of the Chern connection of $(E,h)$ and $\eta$ is a given volume form. \\
\indent Note that on Riemann surfaces ($n=1$) \ref{vectorbundleMAeqn} is just the Hermitian-Einstein equation and hence can be solved if and only if $E$ is Mumford polystable. For line bundles on general manifolds, it is equivalent to the Calabi Conjecture which can be solved if the bundle is ample. Hence it is reasonable to expect that for \ref{vectorbundleMAeqn} to have a solution, one needs some positivity condition on $i\Theta_h$ in addition to a stability condition. We observe that on a general vector bundle, it is not even clear whether there is a solution to \ref{vectorbundleMAeqn} for \emph{some} $\eta>0$ (as opposed to \emph{any} given $\eta>0$). In fact, even if $h$ has Nakano-positive curvature (the strongest positivity assumption), $\frac{(i\Theta_h)^n}{\eta}$ may not be a positive-definite endomorphism. \\
\indent Our first result (Theorem \ref{momentmap} in Section \ref{momentmapsection}) is a moment map interpretation of the vbMA equation.  The corresponding infinite-dimensional symplectic form in Theorem \ref{momentmap} is a K\"ahler form near a connection if and only if the curvature of the connection satisfies a positivity condition that we call MA-positivity. 
\begin{definition}
Let $(E,h)$ be a Hermitian holomorphic vector bundle on an $n$-dimensional complex manifold $M$. An endomorphism-valued $(1,1)$-form $\Theta$  is said to be MA-positively curved if for every non-vanishing endomorphism-valued $(0,1)$-form $a$, the following inequality holds at all points on $M$ where $a \neq 0$. 
\begin{gather}
\displaystyle \sum_{k=0}^n i\mathrm{tr} \Big ( a^{\dag} \left (\frac{i\Theta}{2\pi} \right )^k  a \left (\frac{i\Theta}{2\pi} \right )^{n-k-1} \Big ) >0,
\label{MApositivitydefinitionequation}
\end{gather}
where $a^{\dag}$ is the adjoint of $a$ with respect to $h$ and an $(n,n)$-form is considered to be positive if it is a volume form.
\label{MApositivitydefinition}
\end{definition}

In the case of surfaces, this condition implies Griffiths positivity and is implied by Nakano-and-dual-Nakano positivity (Lemma \ref{MApositivityandothers}). It turns out that $\mathbb{CP}^2$ with the Fubini-Study metric is MA-positive but not Nakano positive (and in fact $\mathbb{CP}^n$ cannot have a Nakano positive metric). \\
\indent Now we define a slope  involving higher Chern characters. 
\begin{definition}
Let $E$ denote a coherent torsion-free sheaf over a smooth projective variety $M$ of dimension $n$. Define the Monge-Amp\`ere slope as $\mu_{MA}(E) = \frac{[\mathrm{ch}_n(E)][M]}{rk(E)}$ where the Chern character class is defined using the Whitney product formula and a finite resolution (a Syzygy) of $E$ by vector bundles. 
\label{defnofslope}
\end{definition}
The Monge-Amp\`ere slope defined above can be computed for coherent subsheaves of vector bundles using only subbundles via resolution of singularities. This strategy is used in Section \ref{stabilitysection}. There is also a corresponding stability condition.
\begin{definition}
A holomorphic vector bundle $E$ over a smooth projective variety $M$ of dimension $n$ is defined to be Monge-Amp\`ere stable (MA-stable) if for every coherent saturated subsheaf $S\subset E$ the following inequality holds.
\begin{gather}
\mu_{MA}(S) < \mu_{MA}(E). \nonumber 
\end{gather}
\label{stabilitycondition}
\end{definition} 
We have the following consequence of solvability of the vbMA equation (proved in Section \ref{stabilitysection}) :
\begin{theorem}
Let $E$ be a holomorphic rank-2 vector bundle on a smooth projective surface $M$. Assume that $\eta>0$ is a given volume form. If there exists a smooth metric $h$ such that $(i\Theta_h)^2  = \eta Id$ and $i\Theta_h$ is Griffiths positive, then the following hold.
\begin{enumerate}
\item \emph{Stability} : If $E$ is indecomposable then $E$ is MA-stable.
\item \emph{Chern class inequality} :  
\begin{gather}
 c_1^2(E) - 4c_2(E) \leq 0 \label{KLBMYtypeinequality}
\end{gather} with equality holding if and only if $\Theta$ is projectively flat.
\end{enumerate}

\label{stabilityandChernclassinequality}
\end{theorem}
\indent In Section \ref{specialcasessection} we look at a few examples where we produce solutions for \emph{some} $\eta$ - $T\mathbb{CP}^n$, Mumford stable bundles, and Vortex bundles. Our last example deals with bundles originally studied by Garc\'ia-Prada \cite{Oscar}. In this example, we dimensionally reduce Equation \ref{vectorbundleMAeqn} to an equation (the Monge-Amp\`ere Vortex equation) for a single function on a Riemann surface and prove a Kobayashi-Hitchin correspondence for it. This theorem may be viewed as the main result of this paper.
\begin{theorem}
Let $(L,h_0)$ be a holomorphic line bundle over a compact Riemann surface $X$ such that its curvature $\Theta_{0}$ defines a K\"ahler form $\omega_{\Sigma}=i\Theta_0$ over $M$. Assume that the degree of $L$ is $1$. Let $r_1,r_2\geq 2$ be two integers, and $\phi \in H^0 (X,L)$ which is not identically $0$. Then the following are equivalent.
\begin{enumerate}
\item Stability : $r_1 >r_2$.
\item Existence : There exists a smooth metric $h$ on $L$ such that the curvature $\Theta_h$ of its Chern connection $\nabla_h$ satisfies the Monge-Amp\`ere Vortex equation.
\begin{gather}
i\Theta_h = (1-\vert \phi \vert_h^2) \frac{\mu \omega_{\Sigma}+i\nabla_h ^{1,0}\phi \wedge \nabla^{0,1} \phi^{\dag_h}}{(2r_2+\vert \phi \vert_h^2)(2+2r_2-\vert \phi \vert_h^2)},
\label{MAVortexeq}
\end{gather}
where $\mu = 2(r_2(r_1+1)+r_1(r_2+1))$ and $\phi^{\dag_h}$ is the adjoint of $\phi$ with respect to $h$ when $\phi$ is considered as an endomorphism from the trivial line bundle to $L$.
\end{enumerate}
Moreover, if a solution $h$ to \ref{MAVortexeq} satisfying $\vert \phi \vert_h^2 \leq 1$ exists, then it is unique.
\label{RSPDEthm}
\label{KobHitcorrespondence}
\end{theorem}
 The existence part of Theorem \ref{KobHitcorrespondence} is proved using the method of continuity. Surprisingly enough, it turns out that the openness of the continuity path as well as uniqueness are the most delicate parts of the proof. Both results use the maximum principle for an appropriately chosen function.\\
\indent It turns out that MA-stability of the corresponding rank-2 Vortex bundle implies $r_1>r_2$ in Theorem \ref{KobHitcorrespondence}. This observation lends evidence to the following conjecture. \\

\emph{The vbMA equation admits a smooth solution if and only if the bundle is MA-stable.}\\

\indent Another observation is that the Vortex bundle is actually Mumford-stable when $r_1>r_2$. This fact ``explains" the KLMBY-type inequality \ref{KLBMYtypeinequality}. Putting this remark and the Griffiths Conjecture (which states that Hartshorne ample bundles admit Griffiths positively curved metrics) together, one is tempted to conjecture the following. \\

\emph{A rank-2 MA-stable indecomposable Hartshorne ample bundle $E$ over a compact complex surface is Mumford semistable}.\\

\emph{Acknowledgements} : The author thanks Richard Wentworth and Indranil Biswas for useful discussions. It is also a pleasure to thank Dror Varolin for constructive criticism as well as encouragement. The author is grateful to the anonymous referee for their comments. This work is partially supported by an SERB grant : ECR/2016/001356. The author is also grateful to the Infosys foundation for the Infosys Young Investigator Award. This work is also partially supported by grant F.510/25/CAS-II/2018(SAP-I) from UGC (Govt. of India).
\section{Moment map interpretation}\label{momentmapsection}
\indent In a manner analogous to that considered by J. Fine \cite{fine} regarding the Calabi Conjecture, equation \ref{vectorbundleMAeqn} can be obtained as the zero of a moment map associated to a Hamiltonian action of a certain gauge group on a space of unitary integrable connections.  More precisely, we have the following theorem.
\begin{theorem}
Let $(E,h_0)$ be a Hermitian complex vector bundle of rank-$r$ with a holomorphic structure given by a unitary integrable connection $\nabla_0$ over a compact complex manifold $M$ of complex dimension $n$. Let $\eta>0$ be an $(n,n)$-form on $M$ such that $\displaystyle \int_M \eta =\frac{n!}{r} \int_M \mathrm{ch}_n (E)$ where $ch$ is the Chern character. Let $\mathcal{A}^{1,1}$ denote the space of unitary integrable connections on $E$. There exists a holomorphic line bundle $\mathbf{Q}$ on an open subset $U^{1,1}\subset\mathcal{A}^{1,1}$ with a unitary Chern connection whose curvature $\Omega$ is a symplectic form on $U^{1,1}$. Moreover, the unitary gauge group $\mathcal{G}$ acts in a Hamiltonian manner on $U^{1,1}$ with a moment map $\mu$. There is a zero of the moment map in the complex gauge orbit of $\nabla_0$ if and only if the following vector bundle Monge-Amp\`ere equation is satisfied for an $h_0$-Hermitian smooth gauge transformation $g$ 
\begin{gather}
\Bigg ( \frac{i(\Theta_0 + \bar{\partial}(g^{-1} \partial_0 g))}{2\pi} \Bigg ) ^n = \eta Id,
\label{vbgenma}
\end{gather}
where $\Theta_0$ is the curvature of $\nabla_0$.
\label{momentmap}
\end{theorem} 
\begin{proof}
We follow some ideas of \cite{fine, Pingad}. The tangent space of $\mathcal{A}^{1,1}$ consists of endomorphism-valued $(0,1)$-forms $a$ such that $\bar{\partial} a=0$. It is clear that the gauge group $\mathcal{G}$ preserves integrability. The Lie algebra $T_I \mathcal{G}$ of the gauge group consists of skew-Hermitian matrices $iH$. For future use we note that  the space $\mathcal{A}^{1,1}$ can be thought of as a complex manifold by thinking of it as a subset of the complex vector space (endowed the conjugate of the usual complex structure) of $\bar{\partial}$ operators satisfying $\bar{\partial}^2=0$. The bundle $\mathbf{Q}$ will be the Quillen determinant bundle of a virtual bundle $\mathcal{E}$ on $M \times A^{1,1}$. \\
\indent Firstly, we find the ``correct" symplectic form $\Omega$.
\begin{lemma}
Let $U^{1,1} \subset \mathcal{A}^{1,1}$ be an open subset consisting of connections $A$ so that the closed 2-form $$\Omega_A (a,b) =\frac{W}{2\pi}\displaystyle \int_M \sum_{k=0}^{n-1}\mathrm{tr}\left (a\Big (\frac{i\Theta_A}{2\pi}\Big) ^k b \Big (\frac{i\Theta_A}{2\pi} \Big ) ^{n-1-k}   \right )$$ (where $W \geq 1$ is any integer) is actually a K\"ahler form. An equivariant moment map $\mu$ corresponding to the symplectic form $\Omega$  is given by the following equation.
\begin{gather}
\mu_A(iH)=W \displaystyle \int_M \mathrm{tr}\left (H \left (\Big (\frac{i\Theta_A}{2\pi} \Big) ^n-\eta Id \right )\right ).
\label{momentmapequation}
\end{gather}
\label{correctsymplecticform}
\end{lemma}
\begin{proof}
The fact that this form $\Omega$ is closed will follow from a later result that it is in fact the curvature of a line bundle. Let $b$ be a skew-Hermitian endomorphism. By definition of the moment map, the variation of $\mu$ at $A$ along $b$ ought to be $-\Omega(id_A H, b)$. Indeed,
\begin{gather}
\delta_b \mu_A (iH)=W\displaystyle \int_M  \sum_{k=0}^{n-1} \mathrm{tr}\left (H\Big (\frac{i\Theta_A}{2\pi} \Big) ^k \frac{id_Ab}{2\pi} \Big (\frac{i\Theta_A}{2\pi} \Big) ^{n-k-1}\right ) \nonumber \\
=-W\displaystyle \int_M  \sum_{k=0}^{n-1} \mathrm{tr}\left (d_A H\Big (\frac{i\Theta_A}{2\pi} \Big) ^k \frac{ib}{2\pi} \Big (\frac{i\Theta_A}{2\pi} \Big) ^{n-k-1}\right ) =-\Omega(id_A H,b).
\end{gather}
Now we prove equivariance. Indeed, 
\begin{gather}
\mu_{g.A}(iH)=\displaystyle \int_M \mathrm{tr}\left (H g\left (\Big (\frac{i\Theta_A}{2\pi} \Big) ^n-\eta Id \right )g^{-1}\right )\nonumber \\
=\mu_{A}(ig^{-1}Hg).
\end{gather}
\end{proof}
\indent It is clear that $\mathcal{G}$ preserves $U^{1,1}$ because $\Theta_{u.A} = u\Theta_A u^{\dag}$ which continues to satisfy the positivity condition (for endomorphism valued $(0,1)$-forms $a$)
\begin{gather}
i\Omega_{u.A} (a^{\dag},a) =\frac{Wi}{2\pi}\displaystyle \int_M \sum_{k=0}^{n-1}\mathrm{tr}\left (a^{\dag}\Big (\frac{i\Theta_{u.A}}{2\pi}\Big) ^k a \Big (\frac{i\Theta_{u.A}}{2\pi} \Big ) ^{n-1-k}   \right )\nonumber \\
= \frac{Wi}{2\pi}\displaystyle \int_M \sum_{k=0}^{n-1}\mathrm{tr}\left (u^{\dag}a^{\dag}u\Big (\frac{i\Theta_{A}}{2\pi}\Big) ^k u^{\dag}au \Big (\frac{i\Theta_{A}}{2\pi} \Big ) ^{n-1-k}   \right )
 >0, \ \forall \ a\neq 0 
\label{positivitycondition}
\end{gather}
if and only if $\Theta_A$ does. \\
\indent  The complexification $\mathcal{G}_{\mathbb{C}}$ acts on $\bar{\partial}$ operators according to $v.A^{0,1} = vA^{0,1}v^{-1} - \bar{\partial}v v^{-1}$. This gives rise to a unitary connection  $v.A = vA^{0,1}v^{-1} +(v^{-1})^{\dag} A^{1,0} v^{\dag}-\bar{\partial}v v^{-1}  + (v^{-1})^{\dag} \partial v^{\dag}$. Every such $v$ can be written uniquely as $v=ug$ where $u$ is $h_0$-unitary and $g$ is $h_0$-Hermitian (the polar decomposition). It is easy to see that $v.A$ is still integrable. The curvature of $v.A$ is (see page 5 of \cite{don2} for instance) $\Theta_{v.A}=v\Theta_A v^{-1} + v\bar{\partial}((g^{\dag}g)^{-1}\partial_A (g^{\dag}g))v^{-1}.$ Therefore, there is a zero of the moment map in the complex gauge orbit of $\nabla_0$ if and only if there is a solution to the vector bundle Monge-Amp\`ere equation. \\
\indent Now take the bundle $E$ equipped with the connection $\nabla_0$ and form the following virtual bundle
\begin{gather}
\mathcal{E} = N \left [\begin{array}{ccccc} 0&0& \ldots & 0& 1\end{array} \right ]^T \left [\begin{array}{ccccc}1& 1& \frac{1}{2!} & \ldots & \frac{1}{(n+1)!} \\
 1& 2& \frac{2^{2}}{2!} & \ldots & \frac{2^{n+1}}{(n+1)!} \\
 \vdots& \vdots& \vdots & \ddots & \vdots \\
 1& (n+2)&  \frac{(n+2)^{2}}{2!} & \ldots & \frac{(n+2)^{n+1}}{(n+1)!} \end{array} \right ]^{-1} \left [ \begin{array}{c} E \\ E^{\otimes 2} \\ E^{\otimes 3} \\ \vdots \\ E^{\otimes (n+2)}\end{array}  \right ],
\end{gather}
where $N$ is a large enough positive integer that clears the denominators. If $\mathcal{E}$ is equipped with the induced connection, then it is easily seen to satisfy 
\begin{gather}
\mathrm{ch}(\mathcal{E})=N \mathrm{tr}\left ( \frac{i\Theta_0}{2\pi}\right )^{n+1}.
\label{almostthere}
\end{gather} 
Consider the virtual bundle $\tilde{\mathcal{E}}=\pi_1^{*} \mathcal{E}$ over $M \times \mathcal{A}^{1,1}$. Define a connection $\mathbb{A}(p,A)=A(p)$ on $\tilde{\mathcal{E}}$. It is not hard to see that this connection defines an integrable $\bar{\partial}$ operator on $\tilde{\mathcal{E}}$. Finally, we have the following lemma that completes the proof of the theorem.
\begin{lemma}
The symplectic form $\Omega$ is the first Chern form of the Quillen metric on the Quillen determinant bundle of $\tilde{\mathcal{E}}$ (equipped with the aforementioned holomorphic structure).
\label{quillenlemma}
\end{lemma}
\begin{proof}
By Theorem 1.27 of \cite{GS} we see that the first Chern form of the Quillen metric of the Quillen determinant of $\tilde{\mathcal{L}}$ is given by 
\begin{gather}
\tilde{\Omega} = \displaystyle \int _M [ch(\tilde{\mathcal{E}})]^{1,1}Td(M).
\label{curvform}
\end{gather}
Consider a surface full of connections in $\mathcal{A}^{1,1}$ defined by $\Phi : X\times \mathbb{R}^2 \rightarrow \mathcal{A}^{1,1}$ as $\Phi (p,x, y) = A-xa-yb$. Therefore, using formulae \ref{almostthere} and \ref{curvform}  we get
\begin{gather}
\tilde{\Omega}_{x=y=0} (a,b) = \Omega_{x=y=0}(a,b),
\end{gather}
thus proving the lemma.
\end{proof}
\end{proof}
\indent Now we have an ``openness" result.
\begin{lemma}
If there is a metric $h$ on an indecomposable holomorphic vector bundle $E$ whose curvature $\Theta_h$ satisfies the vector bundle Monge-Amp\`ere equation $\left ( \frac{i\Theta_h}{2\pi} \right)^n = \eta Id$ for some $\eta>0$ such that $\Theta_h$ is MA-positive then for all volume forms $\alpha$ sufficiently close $\eta$, there exists a smooth solution $h_{\alpha}$ (that is unique up to rescaling by a constant among all solutions in a neighbourhood of $h$) to $\left ( \frac{i\Theta_{h_{\alpha}}}{2\pi} \right)^n = \alpha Id$. 
\label{opennesspositivity}
\end{lemma}
\begin{proof}
\indent If we fix a metric $h_0$, then every other metric is $h=h_0 H$ for some $h_0$-Hermitian invertible endomorphism $H$. Let $\mathcal{B}$ be the Banach manifold consisting of $C^{2,\gamma}$ metrics $h$ satisfying $\displaystyle \int \mathrm{tr}(H) \omega^n =1$ for some fixed K\"ahler form $\omega$.  Therefore the tangent space at $H$ consists of $h_0$-Hermitian endomorphisms $g$ satisfying $\displaystyle \int \mathrm{tr}(g) \omega^n =0$. Let $\mathcal{C}$ be the Banach manifold of $C^{0,\gamma}$ top-form valued endomorphisms $u$ satisfying $\displaystyle r\int \mathrm{tr}(u)  =n! \int\mathrm{ch}_n(E)$. Consider the map $T :\mathcal{B} \rightarrow \mathcal{C}$ given by $T(h)=\left ( \frac{i\Theta_h}{2\pi} \right)^n-\eta Id$. We calculate the derivative below at a point $h_0$ where $h_0$ is MA-positive.
\begin{gather}
DT_{h_0}(g) =\displaystyle \sum_{k=0}^{n-1} \Big (\frac{i\Theta_0}{2\pi}\Big) ^{k} \frac{i\bar{\partial}\partial_h g}{2\pi} \Big (\frac{i\Theta_0}{2\pi}\Big) ^{n-1-k}
\label{eqnlinearisation}
\end{gather}
Using $a=\xi Id$ (where $\xi$ is a $(1,0)$-form) in the definition of MA-positivity, we see that the above operator is elliptic. It is clearly symmetric too. Thus, if we prove that $ker(DT)_h$ is trivial we will be done by the Fredholm alternative. Indeed, suppose $g\in  ker(DT_0)$. Then, multiplying \ref{eqnlinearisation} by $g$, taking trace, and integrating-by-parts we see that
\begin{gather}
0 = \Omega (\partial_0 g, \bar{\partial} g),
\end{gather}
and hence by the MA-positivity condition, $\partial_0 g=0$. Since $g$ is Hermitian, $\nabla_0 g=0$. Diagonalising $g$ we see that since $E$ is indecomposable, $g=\lambda I$ for some real $\lambda$. By normalisation $\lambda=0$.
\end{proof}
\indent Before we proceed further, we recall the various definitions of positivity. 
\begin{definition}
Let $\Theta \in \Lambda^{1,1}(End(E))$ be the curvature endomorphism of the Chern connection of a Hermitian metric $h$ on a holomorphic vector bundle $E$ over a complex manifold $M$. Let $p\in M$ be a point. Then $\Theta(p)$ is said to be
\begin{enumerate}
\item Griffiths positive if $\Theta(p)_{i\bar{j}\beta}^{\alpha}v^{\beta} h_{\alpha \bar{\gamma}} \bar{v}^{\gamma} w^i \bar{w}^j \geq 0$ for all vectors $v\in E_p$ and $w\in T_p ^{1,0}M$, with equality holding iff $v^{\alpha} w^{i} =0 \ \forall \ \alpha,i$.
\item Nakano positive if $\Theta(p)_{i\bar{j}\beta}^{\alpha} a^{i\beta} \bar{a}^{j\gamma} h_{\alpha \bar{\gamma}} \geq 0$ for all tensors $a^{i\beta} \in T^{1,0}_pM \otimes E_p$, with equality holding iff $a^{i\beta}=0 \ \forall i,\beta$.
\item dual Nakano positive if $\Theta(p)_{i\bar{j}\beta}^{\alpha} a^{j\beta} \bar{a}^{i\gamma} h_{\alpha \bar{\gamma}} \geq 0$ for all tensors $a^{i\beta} \in T^{1,0}_pM \otimes E_p$, with equality holding iff $a^{j\beta}=0 \ \forall j,\beta$.
\end{enumerate}
\label{griffithsandnakano}
\end{definition}
The following lemma sheds a little light on the mysterious MA-positivity condition.
\begin{lemma}
The following hold when $n=2$ (on a complex surface).
\begin{enumerate}
\item Nakano-and-dual-Nakano positivity implies MA-positivity.
\item MA-positivity implies Griffiths positivity.
\end{enumerate}
\label{MApositivityandothers}
\end{lemma}
\begin{proof}
Choose a normal holomorphic frame at $p$. From now onwards we work only in this frame at this point.
\begin{enumerate}
\item  We may write the curvature as $i\Theta = A idz^1\wedge d\bar{z}^1 + C idz^2 \wedge d\bar{z}^2 + B idz^1\wedge d\bar{z}^2 + B^{\dag} idz^2 \wedge d\bar{z}^1$ where $A, B, C$ are $r\times r$ complex matrices (with $A=A^{\dag}$, $C=C^{\dag}$). Nakano positivity is easily seen to be equivalent to the $2r \times 2r$ matrix 
\begin{gather}
T = \left [ \begin{array}{cc} A & B^{\dag} \\ B & C \end{array} \right ]
\label{Tmatrix}
\end{gather}
being positive-definite. Likewise, dual Nakano positivity is equivalent to the positivity of 
\begin{gather}
T' = \left [ \begin{array}{cc} A & B \\ B^{\dag} & C \end{array} \right ].
\label{T'matrix}
\end{gather}
Now suppose $a^{\dag}$ is an $r\times r$ matrix of $(1,0)$ forms given by $a^{\dag}=\alpha dz^1 + \beta dz^2$ where $\alpha, \beta$ are $r\times r$ matrices of complex numbers. Assume that $\alpha_l, \beta_l$ and $\alpha^l, \beta^l$ are the $l^{th}$ columns and rows of $\alpha, \beta$ respectively. Then we see that
\begin{gather}
\frac{\mathrm{tr} \left (ia^{\dag} \left ( i\Theta \right ) a \right )+\mathrm{tr} \left (ia^{\dag}  a  \left ( i\Theta \right )\right )}{idz^1d\bar{z}^1 idz^2 d\bar{z}^2} = \mathrm{tr}(\alpha C \alpha^{\dag})+\mathrm{tr}(\alpha \alpha^{\dag}C)+\mathrm{tr}(\beta A \beta^{\dag})+\mathrm{tr}(\beta  \beta^{\dag} A) \nonumber \\
-\mathrm{tr}(\alpha B^{\dag} \beta^{\dag})-\mathrm{tr}(\alpha  \beta^{\dag} B^{\dag}) -\mathrm{tr}(\beta B \alpha^{\dag}) - \mathrm{tr}(\beta  \alpha^{\dag} B).
\end{gather}
Note that $\mathrm{tr}(\alpha C \alpha^{\dag}) = \alpha_1 C \alpha_1 ^{\dag} + \alpha_2 C \alpha_2 ^{\dag}$ and likewise $\mathrm{tr}(\alpha^{\dag} B \beta) =  \alpha_1^{\dag} B \beta_1 + \alpha_2^{\dag} B \beta_2 $. Therefore using the assumption that $T$ and $T'$ are positive-definite,
\begin{gather}
\frac{\mathrm{tr} \left (ia^{\dag} \left ( i\Theta \right ) a \right )+\mathrm{tr} \left (ia^{\dag}  a  \left ( i\Theta \right )\right )}{idz^1d\bar{z}^1 idz^2 d\bar{z}^2} = \displaystyle \sum_{l=1}^2 \left [ \begin{array}{cc} \beta_l^{\dag} & -\alpha^{\dag}_l \end{array} \right ] \left [ \begin{array}{cc} A & B^{\dag} \\ B & C \end{array} \right ] \left [\begin{array}{c} \beta_l \\ -\alpha_l \end{array} \right ] \nonumber \\
+ \displaystyle \sum_{l=1}^2 \left [ \begin{array}{cc} \beta^l & -\alpha^l \end{array} \right ] \left [ \begin{array}{cc} A & B \\ B^{\dag} & C \end{array} \right ] \left [\begin{array}{c} (\beta^l)^{\dag} \\ -(\alpha^l)^{\dag} \end{array} \right ] \geq 0,  
\label{nakposimplies}
\end{gather}
with equality holding if and only if $a=0$. Hence Nakano-and-dual-Nakano positivity implies MA-positivity for complex surfaces.
\item Griffiths positivity means that $\Theta (\xi, \bar{\xi})$ is a positive-definite matrix for all co-vectors $\xi\neq 0$. Given a $\xi$, choose coordinates so that $\xi=\frac{\partial}{\partial z^1}$. So we need to prove that $A$ is a positive-definite matrix, i.e., $v^{\dag} A v >0$ for all $v \neq 0$. Indeed, choose $a=dz^2 \beta$. Then 
\begin{gather}
\frac{\mathrm{tr} \left (ia^{\dag} \left ( i\Theta \right ) a \right )+\mathrm{tr} \left (ia^{\dag}  a  \left ( i\Theta \right )\right )}{idz^1d\bar{z}^1 idz^2 d\bar{z}^2}= \mathrm{tr}(\beta A \beta^{\dag})+\mathrm{tr}(\beta^{\dag} A \beta) >0,
\end{gather}
if $\beta \neq 0$. Suppose $\beta_{ij} =  v_i \bar{v}_j$. Then $\beta = \beta^{\dag}$ and hence
\begin{gather}
0< \mathrm{tr}(\beta A \beta^{\dag})+\mathrm{tr}(\beta^{\dag} A \beta) = 2\sum_{i} v_i \bar{v}_j A_{jk} v_k \bar{v}_i = 2 \vert v \vert^2 v^{\dag} A v.
\end{gather}  
Hence for surfaces, MA-positivity implies Griffiths positivity.
\end{enumerate}
\end{proof}
\indent In order to prove existence results, one typically uses the method of continuity. The following potentially useful lemma shows that for certain kinds of continuity paths Griffiths positivity is preserved.
\begin{lemma}
Let $h(t)$ (where $t\in[0,1]$) denote a path of metrics on a holomorphic vector bundle $E$ over a compact complex surface $M$ such that $\Theta_{t=0}$ is Griffiths positive, and $\frac{(i\Theta_{t})^2}{\eta}>0$ as positive-definite endomorphisms (where $\eta$ is any fixed volume form). Then $\Theta_t$ is Griffiths positive for all $t\in[0,1]$.
\label{gpispreserved}
\end{lemma}
\begin{proof}
As before, it is enough to prove that $\Theta_t (\xi, \bar{\xi})$ is positive-definite where by a change of coordinates we may assume that $\xi=\frac{\partial}{\partial z^1}$. Let $T$ be the first $t$ such that $v^{\dag}\Theta_t (\xi, \bar{\xi})v=0$ for some $v$. Using the same notation as before, note that
\begin{gather}
(i\Theta)^2 = idz^1 d\bar{z}^1 dz^2 d\bar{z}^2 (AC+CA-BB^{\dag}-B^{\dag}B)>0
\end{gather}
Thus, $v^{\dag}(AC+CA)v>v^{\dag}(BB^{\dag}+B^{\dag} B)v$. Since $v^{\dag}\Theta_t (\xi, \bar{\xi})v=0$, we see that $Av=0$ ($A$ is positive semidefinite). So $0>\Vert B v\Vert^2 + \Vert B^{\dag} v\Vert^2$ which is impossible. Hence Griffiths positivity is preserved.
\end{proof}
\indent It is easy to see that the linearization of the vector bundle Monge-Amp\`ere equation is elliptic if and only if Griffiths positivity holds. So the point of Lemma \ref{gpispreserved} is that for most reasonable continuity paths, ellipticity is preserved.
\section{Consequences of existence}\label{stabilitysection}
\indent In this section we prove two consequences of existence of positively curved solutions to the vbMA equation, and thus prove Theorem \ref{stabilityandChernclassinequality}.\\
\subsection{MA Stability}
\begin{lemma}
Let $E$ denote an indecomposable holomorphic rank-2 vector bundle on a smooth projective surface $M$. Le $\eta>0$ be a given volume form. If there exists a smooth metric $h$ such that $(i\Theta_h)^2  = \eta Id$ and $tr(i\Theta_h)$ is positive, then for every holomorphic subbundle $S$, $\mu_{MA}(S) < \mu_{MA}(E)$.
\label{forsubbundles}
\end{lemma}
\begin{proof}
Suppose $Q$ is the quotient line bundle and $\beta$ is the second fundamental form. Then the curvature looks like the following
\begin{gather}
\Theta = \left [ \begin{array} {cc} \Theta_S-\beta \wedge \beta^{\dag} & \nabla^{1,0}\beta \\ -\nabla^{0,1} \beta^{\dag} & \Theta_Q-\beta^{\dag}\wedge \beta \end{array}\right ].
\end{gather}
Since $(i\Theta)^2 =\eta Id$, we see that 
\begin{gather}
\mathrm{tr}(i^2(\Theta_S - \beta \wedge \beta^{\dag})^2) - i^2\mathrm{tr}(\nabla^{1,0}\beta \wedge \nabla^{0,1}\beta^{\dag}) = \eta \nonumber \\
\Rightarrow 2(2\pi)^2ch_2(S)-2i^2\mathrm{tr}(\Theta_S \beta \wedge \beta^{\dag})-i^2\mathrm{tr}(\nabla^{1,0}\beta \wedge \nabla^{0,1}\beta^{\dag}) = \eta
,\label{thetasequation}
\end{gather}
and likewise,
\begin{gather}
\mathrm{tr}(i^2(\Theta_Q - \beta^{\dag} \wedge \beta)^2) - i^2\mathrm{tr}(\nabla^{1,0}\beta \wedge \nabla^{0,1}\beta^{\dag}) = \eta \nonumber \\
\Rightarrow 2(2\pi)^2 ch_2(Q)-2i^2\mathrm{tr}(\Theta_Q \beta^{\dag} \wedge \beta)-i^2\mathrm{tr}(\nabla^{1,0}\beta \wedge \nabla^{0,1}\beta^{\dag}) = \eta.
\label{thetaqequation}
\end{gather}
Using \ref{thetasequation} and \ref{thetaqequation}, and integrating we get,
\begin{gather}
\displaystyle \int_M (2\pi)^2 ch_2(S) + \frac{1}{2}\int_M i^2\mathrm{tr}((\Theta_Q+\Theta_S)\beta^{\dag}\wedge \beta) = \frac{1}{2}\int_M (2\pi)^2ch_2(E)\nonumber \\
\Rightarrow \displaystyle \int_M (2\pi)^2 ch_2(S) + \frac{1}{2}\int_M i2\pi\mathrm{tr}(c_1(E)\beta^{\dag}\wedge \beta) = \frac{1}{2}\int_M (2\pi)^2 ch_2(E).
\label{finalgriffithspositive}
\end{gather}
The given positivity condition is $c_1(E)>0$ and hence \ref{finalgriffithspositive} implies that $\mu_{MA}(S) \leq \mu_{MA}(E)$ with equality holding if and only if $E$ is decomposable.  Hence $\mu_{MA}(S) < \mu_{MA}(E)$.
\end{proof}
\indent In order to prove a version of Lemma \ref{forsubbundles} for coherent saturated subsheaves, we need to do more work (akin to \cite{richard, adam, sibley}). 

\begin{proposition}
If $(E,h)$ is an indecomposable Hermitian holomorphic rank-2 vector bundle on a smooth projective surface $M$ such that $(i\Theta_h)^2=\eta Id$ where $\eta>0$, and $\mathrm{tr}(i\Theta_h)$ is positive, then $E$ is MA-stable.
\label{stability}
\end{proposition}
\begin{proof}
Let $S\subset E$ be a coherent saturated subsheaf. By assumption, the singularities of $S$ are points (codimension $2$). We adapt the discussion following Corollary 4.2 in \cite{sibley}. Essentially, one takes the singular locus of $S$, blows it up finitely many times to get $\pi : \tilde{M} \rightarrow M$, and takes $\tilde{S}$ to be the saturation of $\pi^{*}S$ in $\pi^{*}(E)$. It is a subbundle of $E$ such that $T=\tilde{S}/\pi^{*}S$ is a torsion sheaf supported on points $\sum_i m_{Z_{i,0}}Z_{i,0}$ (by the reasoning in the proof of Proposition 4.3 in \cite{sibley}).  At this point we appeal to Proposition 3.1 of \cite{richard} which shows that $\mathrm{ch}_2(T)=PD(\displaystyle \sum_{i}m_{Z_{i,0}} Z_{i,0})$ where $PD$ is the Poincar\'e dual. Therefore, 
\begin{gather}
\pi^{*}\mathrm{ch}_2(S) = \mathrm{ch}_{2}(\tilde{S})-PD(\displaystyle \sum_{i}m_{Z_{i,0}} Z_{i,0}). 
\label{relationbetweentildeSandS}
\end{gather}
From \ref{relationbetweentildeSandS} we see that $\mu_{MA}(S) \leq \mu_{MA}(\tilde{S})$. By Lemma \ref{forsubbundles}, $\mu_{MA}(\tilde{S})<\mu_{MA}(E)$. Hence $E$ is MA-stable.
\end{proof}
\subsection{Chern class inequality}
\textbf{}\\
\indent We prove the following KLBMY-type inequality thus completing the proof of Theorem \ref{stabilityandChernclassinequality}.
\begin{proposition}
Let $(E,h)$ be a Hermitian rank-2 holomorphic vector bundle on a compact complex surface such that $i\Theta_h>0$ in the sense of Griffiths, and  $(i\Theta_h)^2 = \eta Id$ where $\eta>0$ is a volume form and $\Theta_h$ is the curvature of the Chern connection of $h$. Then $c_1^2(E) - 4c_2(E) \leq 0$ with equality holding if and only if $\Theta$ is projectively flat.
\label{KLnew}
\end{proposition}
\begin{proof}
Choose a holomorphic frame near $p$ which is also orthonormal at $p$. Then $\Theta_{12} = -\bar{\Theta}_{21}$. Moreover, the Griffiths positivity of $i\Theta_h$ implies that we can choose coordinates near $p$ so  that at $p$, $i\Theta_{11}(p)=idz^1 \wedge d\bar{z}^1 + idz^2\wedge d\bar{z}^2$ and $i\Theta_{22}(p)=i\lambda_1 dz^1 \wedge d\bar{z}^1+i\lambda_2 dz^2 \wedge d\bar{z}^2$ where $\lambda_1, \lambda_2 >0$. Also, $\eta=fidz ^1 \wedge d\bar{z}^1\wedge idz^2 \wedge d\bar{z}^2$. Thus,
\begin{gather}
\eta (p) Id = (i\Theta_h)^2 (p) \nonumber \\
 = \left [ \begin{array}{cc} 2idz^1 \wedge d\bar{z}^1 \wedge dz^2 \wedge d\bar{z}^2-\Theta_{12} \bar{\Theta}_{12}(p) & \Theta_{11}\Theta_{12} (p)+\Theta_{12} \Theta_{22}(p)  \\ -\Theta_{11} \bar{\Theta}_{12}(p)-\bar{\Theta}_{12} \Theta_{22} (p) &  2i \lambda_1 \lambda_2 dz^1 \wedge d\bar{z}^1 \wedge dz^2 \wedge d\bar{z}^2-\Theta_{12} \bar{\Theta}_{12} (p)\end{array} \right ].
\label{eqnatp}
\end{gather}
Therefore, 
\begin{gather}
1=\lambda_1 \lambda_2, \nonumber \\
 (\Theta_{12})_{22}(p)(\lambda_1+1)=-(\Theta_{12})_{11}(p) (\lambda_2+1), \ and \ \nonumber \\
2+(\Theta_{12})_{11}(\bar{\Theta}_{12})_{22}(p)+(\Theta_{12})_{22}(\bar{\Theta}_{12})_{11}(p)-\vert (\Theta_{12})_{12}^2 \vert^2 (p) -\vert (\Theta_{12})_{21} \vert^2 (p) = f(p). 
\label{beforefinalineq}
\end{gather}
Substituting the second equation in \ref{beforefinalineq} in the third and using the first equation we get
\begin{gather}
2-2\frac{\vert(\Theta_{12})_{11} \vert^2(p)}{\lambda_1} -\vert (\Theta_{12})_{21} \vert^2 (p)-\vert (\Theta_{12})_{12} \vert^2 (p) = f(p).
\label{beffinsimp}
\end{gather}
Now,
\begin{gather}
(-i)^2(c_1^2(p)- 4c_2(p)) = (\Theta_{11}(p)+\Theta_{22}(p))^2 - 4 (\Theta_{11}\Theta_{22}+\Theta_{12} \bar{\Theta}_{12})(p) \nonumber \\
= dz^1\wedge d\bar{z}^1 \wedge dz^2 \wedge d\bar{z}^2 \Big ( 2(1+\lambda_1)(1+\lambda_2) - 4 (\lambda_1+ \lambda_2)  +4(f(p)-2)\Big).
\label{KLsemifinal}
\end{gather}
Using \ref{beffinsimp} and \ref{KLsemifinal} we get the following.
\begin{gather}
\frac{c_1^2(p)- 4c_2(p)}{idz^1\wedge d\bar{z}^1 \wedge idz^2 \wedge d\bar{z}^2}= 4-2\left(\lambda_1+\frac{1}{\lambda_1}\right)+4(-\vert (\Theta_{12})_{12}^2 \vert^2 (p) -\vert (\Theta_{12})_{21} \vert^2 (p)-2\frac{\vert(\Theta_{12})_{11} \vert^2(p)}{\lambda_1}) \nonumber \\
\leq 0,
\label{KLfinal}
\end{gather}
with equality holding if and only if $\Theta = \omega I$ where $\omega^2 = \eta$ (which can be solved because it is the Calabi conjecture).
\end{proof}
\section{Special cases}\label{specialcasessection}
\indent In this section we discuss some examples. 
\subsection{Projective space}
\indent The projective space $\mathbb{CP}^n$ with the Fubini-Study metric $\omega_{FS}=\frac{i}{2\pi}\pbp \ln(1+\vert z \vert^2)$ (where $\vert z \vert^2 = \sum_i \vert z^i \vert^2$) does satisfy $(\frac{i}{2\pi}\Theta_{FS})^n = \lambda \omega_{FS}^n Id$ where $\lambda$ is a constant. (So, at least there is one metric on $T\mathbb{CP}^n$ whose curvature's power is proportional to identity.) Indeed, the curvature of the Fubini-Study metric is given by $\Theta_{FS} = \bar{\partial}(\bar{H}^{-1}\partial \bar{H})$ where $iH_{ij}dz^i \wedge d\bar{z}^j=\omega_{FS}$. Therefore,
\begin{gather}
\bar{H}_{ij} = \frac{1}{2\pi (1+\vert z \vert^2)} \left [\delta_{ij}-\frac{z^i \bar{z^j}}{1+\vert z \vert^2} \right ] =\frac{1}{2\pi (1+\vert z \vert^2)} \tilde{H}_{ij} \nonumber \\
\frac{i\Theta_{FS}}{2\pi} = \omega_{FS} + \frac{i}{2\pi} \bar{\partial}(\tilde{H}^{-1} \partial \tilde{H}). 
\end{gather}
We calculate $\bar{\partial}(\tilde{H}^{-1} \partial \tilde{H})$ : 
\begin{gather}
\tilde{H}^{-1}_{ik} = \delta_{ik}+z^i \bar{z}^k \nonumber \\
\partial \tilde{H}_{kj} = -\frac{dz^k \bar{z}^j}{1+\vert z \vert^2} + \frac{(\displaystyle \sum_l \bar{z}^l dz^l)z^k \bar{z}^j}{(1+\vert z \vert^2)^2} \Rightarrow (\tilde{H}^{-1} \partial \tilde{H})_{ij} = -\frac{dz^i \bar{z}^j}{1+\vert z \vert^2} \nonumber \\
 \Rightarrow \bar{\partial}(\tilde{H}^{-1} \partial \tilde{H})_{ij} = \frac{dz^i \bar{dz}^j}{1+\vert z \vert^2} +\frac{(\sum_k z^k \bar{dz}^k)dz^i\bar{z}^j}{(1+\vert z \vert^2)^2}
\label{FScalc}
\end{gather}
Since the Fubini-Study metric is symmetric, i.e., there is a transitive group of isometries, it is enough to prove that $\left (\frac{i\Theta_{FS}}{2\pi}\right)^n = \lambda \omega_{FS}^n$ at the point $z=0$. Indeed, at $z=0$, 
\begin{gather}
\left (\frac{i\Theta_{FS}}{2\pi}\right)_{\alpha \beta} (z=0)= \frac{i}{2\pi} \delta_{ij}dz^i \wedge d\bar{z}^j \delta_{\alpha \beta} + \frac{i}{2\pi} dz^{\alpha} \wedge d\bar{z}^{\beta} \nonumber \\
\Rightarrow \left (\frac{i\Theta_{FS}}{2\pi}\right)^n _{\alpha \beta} (z=0) = \displaystyle \sum _{r=0}^{r=n-1} {n \choose r} \omega_{FS}^r (-\omega_{FS})^{n-r-1} dz^{\alpha} \wedge d\bar{z}^{\beta} +\omega_{FS}^n\nonumber \\
= \omega_{FS}^n \delta_{\alpha \beta}+(n-1)!\left (\frac{i}{2\pi} \right)^{n-1} \sum_{j}(dz^1\wedge d\bar{z}^1 \ldots \widehat{dz}^j \wedge \widehat{d\bar{z^j}}\wedge\ldots)\wedge  \sum_{r=0}^{n-1} {n \choose r} (-1)^{n-r-1} dz^{\alpha} \wedge d\bar{z}^{\beta} \nonumber \\
= \omega_{FS}^{n} \delta_{\alpha \beta} (1+ \sum_{r=0}^{n-1} {n \choose r} (-1)^{n-r-1}) = 2\omega_{FS}^n.
\end{gather}
\indent Actually, the Fubini-Study metric on $\mathbb{CP}^2$ satisfies the positivity condition in Section \ref{momentmapsection} and hence, for small perturbations $\eta$ of $\lambda \omega_{FS}^2$ there exist solutions of the vector bundle Monge-Amp\`ere equation. Indeed, by symmetry it suffices once again to verify this at $z=0$. Suppose $[a]_{\alpha \beta}=a_{\alpha \beta,\mu} dz^{\mu}$ is a $2\times 2$ matrix of $(1,0)$-forms.
\begin{gather}
 \mathrm{tr}\Bigg(ia \left (\frac{i\Theta_{FS}}{2\pi}\right) (z=0)  a^{\dag} \Bigg)+ \mathrm{tr}\Bigg( \left (\frac{i\Theta_{FS}}{2\pi}\right) (z=0) ia a^{\dag} \Bigg) = 2i\omega_{FS}a_{\alpha \beta} \wedge \bar{a}_{\alpha \beta} \nonumber \\ +  \frac{i}{2\pi} dz^{\alpha}\wedge d\bar{z}^{\beta} ia_{\beta \gamma} \wedge \bar{a}_{\alpha \gamma} + \frac{i}{2\pi} dz^{\beta}\wedge d\bar{z}^{\gamma} ia_{\alpha \beta} \wedge \bar{a}_{\alpha \gamma} \nonumber \\
= \frac{i}{2\pi}dz^1 \wedge d\bar{z}^1 \wedge idz^2 \wedge d\bar{z}^2 \Big(2\displaystyle \sum_{\alpha, \beta} (\vert a_{\alpha \beta,2} \vert^2 + \vert a_{\alpha \beta,1} \vert^2) \nonumber \\+ \vert a_{1\gamma,2} \vert^2 + \vert a_{2\gamma,1} \vert^2 - a_{2\gamma,2} \bar{a}_{1\gamma,1}- a_{1\gamma,1} \bar{a}_{2\gamma,2} + \vert a_{\alpha 1,2} \vert^2+ \vert a_{\alpha 2,1} \vert^2-a_{\alpha 1,2} \bar{a}_{\alpha 2, 1}-a_{\alpha 2,1} \bar{a}_{\alpha 1,2} \Big) \nonumber \\
 \geq 0,
\end{gather}
with equality holding if and only if $a_{\alpha \beta, \gamma} =0 \ \forall \ \alpha, \beta, \gamma$.
\subsection{Mumford stable bundles} 
\indent Suppose a holomorphic vector bundle $E$ over a compact complex surface is Mumford stable with respect to a polarisation $c_1(L)$. Then we prove that $E\otimes L^k$ admits solutions to the vector bundle Monge-Amp\`ere equation (for a sufficiently large $k$ depending on the right hand side):
\begin{theorem}
Assume that a rank-r holomorphic vector bundle $E$ over a compact complex surface $M$ is Mumford stable with respect to an ample class $c_1(L)$. Given a volume form $\eta>0$ such that $\displaystyle \int_M \eta = \int_M c_1(L)^2$, there exists a positive integer $k$ such that $E\otimes L^k$ has a metric $h_k$, the curvature $\Theta_k$ of whose Chern connection is Griffiths positive and satisfies
\begin{gather}
\left ( \frac{i\Theta_k}{2\pi} \right) ^2 = \eta_k = \eta \frac{\displaystyle \int_M 2\mathrm{ch}_2(E\otimes L^k)}{\displaystyle \int_M r\eta} Id = k^2 \eta Id+   k\frac{\displaystyle \int_M 2c_1(E) c_1(L)}{\displaystyle \int_M r\eta}\eta Id +  \frac{\displaystyle \int_M 2\mathrm{ch}_2(E)}{\displaystyle \int_M r\eta} \eta Id.
\end{gather}
\label{nearHE}
\end{theorem}
\begin{proof}
By Yau's solution of the Calabi conjecture, there exists a metric $h_0$ on $L$ such that its first Chern form $\omega_0$ satisfies $\omega_0^2 =\eta$. Let $g_0$ be a Hermitian-Einstein metric on $E$ (which exists by the Uhlenbeck-Yau theorem \cite{UY}) with curvature $\Theta_0$. Every other metric on $E$ is of the form $g=g_0H$ where $H$ is a $g_0$-Hermitian endomorphism of $E$. The curvature of $g$ is $\Theta = \Theta_0 + \bar{\partial} (H^{-1} \partial_0 H)$ (here the matrix for the metric is such that $\langle s,t\rangle = \overline{s^{\dag} H t}$). Consider the following family of equations parametrised by $\hbar=\frac{1}{k}$ for a Hermitian endomorphism $H_{\hbar}$ satisfying $\int_M tr(H_{\hbar}) vol_{\omega_0} =r \int_M c_1(L)$.
\begin{gather}
\left ( \frac{i}{2\pi}(\hbar \Theta_0 + \omega_0 +  \hbar \bar{\partial} (H_{\hbar}^{-1} \partial_0 H_{\hbar})  ) \right )^2 = \eta Id +  \hbar\frac{\displaystyle \int_M 2c_1(E) c_1(L)}{\displaystyle \int_M r\eta}\eta Id +  \hbar^2\frac{\displaystyle \int_M 2\mathrm{ch}_2(E)}{\displaystyle \int_M r\eta} \eta Id.
\label{perturbedequation}
\end{gather}
The above equation \ref{perturbedequation} has a solution $H_{0} = Id$ at $\hbar=0$. The implicit function theorem proves that for small $\hbar$ (i.e. large $k$) equation \ref{perturbedequation} has a solution provided the linearization at $\hbar=0$ is an isomorphism. However, it is not prudent to linearise the above equation directly because the zeroeth order term is $0$. Therefore we consider the linearization of $$T(\hbar, H) = \frac{1}{\hbar} \Bigg ( \left ( \frac{i}{2\pi}(\hbar \Theta_0 + \omega_0 +  \hbar \bar{\partial} (H^{-1} \partial_0 H)  ) \right )^2 - \eta Id -  \hbar\frac{\displaystyle \int_M 2c_1(E) c_1(L)}{\displaystyle \int_M r\eta}\eta Id -  \hbar^2\frac{\displaystyle \int_M 2\mathrm{ch}_2(E)}{\displaystyle \int_M r\eta} \eta Id \Bigg )$$
at $\hbar=0$. Indeed, $T(0,Id)=0$. The linearization is $$DT_{(0,I)}(h) = 2\omega_0 \wedge \bar{\partial} \partial_0 h. $$ This is of course a self-adjoint map and hence by the Fredholm alternative we need to prove that it has trivial kernel. Indeed, if $DT(0,h)=0$, then $\displaystyle \int_M \omega_0 \wedge tr(h\bar{\partial} \partial_0 h)=0$. Integration-by-parts allows us to conclude that $h$ is parallel. Therefore, the eigenspaces of $h$ are parallel transported to form holomorphic subbundles of $E$. Since $E$ is stable, it is indecomposable and hence $h=cId$. By the normalisation condition, $c=0$. Hence we are done.
\end{proof} 
\subsection{The Vortex bundle}
\indent In \cite{Oscar} a construction of a rank-2 holomorphic bundle over a product of Riemann surfaces $\Sigma \times \mathbb{CP}^1$ was given. This bundle had a high degree of symmetry that was exploited to reduce its Mumford stability to a simple Chern classes inequality. Moreover, the Hermitian-Einstein equation could be reduced to a single PDE for a single function on a Riemann surface (to which the Kazdan-Warner theory could be applied). For the remainder of this paper, we study these Vortex bundles in the context of the vbMA equation. In particular, we prove Theorem \ref{KobHitcorrespondence} in this section. To have notation consistent with the existing literature on these bundles, we drop the $2\pi$ in several of our definitions below (especially in the Fubini-Study metric). \\ 
\indent Consider a genus $g$ compact Riemann surface $\Sigma$ endowed with a metric whose $(1,1)$-form is $\omega_{\Sigma}=i\Theta_0$ where $\Theta_0$ is the curvature of a metric $h_0$ on a degree one line holomorphic bundle $L$. Let $\mathbb{CP}^1$ be endowed with the Fubini-study metric $\omega_{FS} = \frac{i dz\wedge d\bar{z}}{(1+\vert z \vert^2)^2}$ which is the curvature of a metric $h_{FS}$ on $\mathcal{O}(1)$. Consider the rank-2 bundle  $$V=\pi_1^{*}((r_1+1)L)\otimes \pi_2^{*}(r_2 \mathcal{O}(2)) \oplus \pi_1 ^{*}(r_1 L) \otimes \pi_2 ^{*} ((r_2+1)\mathcal{O}(2)),$$ where $r_1,r_2 \geq 2$. \\
\indent Suppose $h$ is a smooth metric on $L$ and $f_2$ is a smooth positive function on $\Sigma$. Put a metric $H=h_1 \oplus  g_2$ on $V$ where $h_1=\pi_1^{*}(h f_2 h_0 ^{r_1}) \otimes \pi_2^{*}(h_{FS} ^{2r_2})$ is a metric on $\pi_1^{*}((r_1+1)L)\otimes \pi_2^{*}(r_2 \mathcal{O}(2))$ and $g_2=\pi_1^{*}(f_2 h_0^{r_1}) \otimes \pi_2^{*}(h_{FS}^{2r_2+2})$ is a metric on $ \pi_1^{*}(r_1L)\otimes \pi_2^{*}((r_2+1) \mathcal{O}(2))$. \\
\indent  Using a holomorphic section $\phi \in H^{0}(X,L)$ endow $V$ with a holomorphic structure through the second fundamental form $\beta =\pi_1^{*}\phi \otimes \pi_2 ^{*}\zeta \in H^{1}(X,\pi_1^{*} L\otimes \pi_2^{*} \mathcal{O}(-2)) \simeq H^0 (\Sigma, L)$ where $\zeta = \frac{\sqrt{8\pi } dz}{(1+\vert z \vert^2)^2} \otimes d\bar{z}$. \\
\subsection{Dimensional reduction to the  Monge-Amp\`ere Vortex equation}
\indent We now reduce the vbMA equation for the Vortex bundle $V$ for certain symmetric right-hand-sides to a single equation on $\Sigma$. To this end, the Chern connection of $(V,H)$ is given by the following expression.  
\begin{gather}
A=\left ( \begin{array}{cc} A_{h_1} & \beta \\ 
 -\beta^{\dag_h}  &  A_{g_2} \end{array} \right )
\end{gather}
Its curvature is
\begin{gather}
\Theta = \left ( \begin{array}{cc} \Theta_{h_1}-\beta \wedge \beta^{\dag_h} & \nabla ^{1,0} \beta \\ 
 -\nabla ^{0,1} \beta^{\dag_h}  &  \Theta_{g_2} - \beta^{\dag} \wedge \beta \end{array} \right ) \nonumber \\
= \left ( \begin{array}{cc} \Theta_{h}+\Theta_{f_2}-ir_1 \omega_{\Sigma}-2ir_2 \omega_{FS} -i\vert \phi \vert_h^2 \omega_{FS} & \nabla ^{1,0} \phi \wedge \pi_2 ^{*} \zeta \\ 
 -\nabla ^{0,1} \phi^{\dag_h} \wedge \pi_2 ^{*} \zeta^{\dag} &  \Theta_{f_2} -ir_1 \omega_{\Sigma}-(2r_2+2)i \omega_{FS} + i  \vert \phi \vert_h^2 \omega_{FS}
\end{array} \right ).
\label{curvatureformula}
\end{gather}
For future use we record that
\begin{gather}
\nabla ^{1,0} \beta \wedge \nabla ^{0,1} \beta^{\dag_h} = -i \nabla ^{1,0} \phi \wedge \nabla ^{0,1}\phi^{\dag_h} \wedge \omega_{FS}.
\label{squarederivativeofsecondfundform}
\end{gather}
\indent We note that rescaling the metric $h_0$ to $h_0 A$ does not change the curvature $i\Theta_0 = \omega_{\Sigma}$. Assume that $h_0$ has been rescaled so that $\vert \phi \vert_0^2 <\frac{1}{2}$. Now we dimensionally reduce the vbMA equation to the Monge-Amp\`ere Vortex equation.
\begin{theorem}
 Suppose there is a smooth metric $h$ on $L$ satisfying
$$\vert \phi\vert_h^2 \leq 1,$$
and solving the following equation. 
\begin{gather}
 i\Theta_h = (1- \vert \phi \vert_h^2)\frac{\zeta+i\nabla ^{1,0} \phi \wedge \nabla^{0,1} \phi^{\dag_h}}{(2r_2+ \vert \phi \vert_h^2)(2+2r_2- \vert \phi \vert_h^2)}, \nonumber
\end{gather}
where $\zeta>0$ is a given $(1,1)$-form on $\Sigma$ satisfying
$$\displaystyle \int _{\Sigma} \zeta = 2(r_1(r_2+1)+r_2(r_1+1)).$$ \\
\indent Then there is a smooth Griffiths positively curved metric $H$ on the Vortex bundle $V$ whose curvature $\Theta$ satisfies the vbMA equation :
\begin{gather}
 (i\Theta)^2 = \pi_1^{*}\zeta\wedge \pi_2^{*}\omega_{FS} I.
\label{symmetricvbma}
\end{gather}
\label{dimensionalreductionthm}
\end{theorem}
\begin{proof}
  Without loss of generality, we may choose the K\"ahler form $\omega_{\Sigma}$ on $\Sigma$ to satisfy $$\mu \omega_{\Sigma} = \zeta,$$
where $\mu$ is a constant given by the following expression.
$$\mu= (2\pi)^2 \int_{X\times \mathbb{CP}^1} ch_2(V) = 2(r_2(r_1+1)+r_1(r_2+1))>0.$$ \\
\indent  Substituting \ref{curvatureformula} in the vbMA \ref{symmetricvbma}, we get the following equations. 
\begin{gather}
\left [\begin{array}{cc} (i\Theta_{h_1}-i\beta \wedge \beta^{\dag_h})^2+\nabla ^{1,0}\beta \wedge \nabla ^{0,1} \beta^{\dag_h} & 0 \nonumber \\ 0& (i\Theta_{g_2}-i\beta^{\dag_h}\wedge \beta)^2+ \nabla ^{0,1} \beta^{\dag_h}\wedge \nabla ^{1,0}\beta\end{array} \right ] = \eta Id.
\label{genma}
\end{gather}
The above equation can be simplified to yield a system of equations.
\begin{align}
2(i\Theta_h + i\Theta_{f_2}+r_1 \omega_{\Sigma})(2r_2+ \vert \phi \vert_h^2)-i \nabla ^{1,0}\phi \wedge \nabla ^{0,1}\phi^{\dag_h} &= \mu \omega_{\Sigma},\label{genma1}
\end{align}
and
\begin{align}
\ \ 2(i\Theta_{f_2}+r_1\omega_{\Sigma})(2r_2+2- \vert \phi \vert_h^2)-i \nabla ^{1,0}\phi \wedge \nabla ^{0,1}\phi^{\dag_h} &= \mu \omega_{\Sigma}.
\label{genma2}
\end{align}

\indent From \ref{genma1} and \ref{genma2} we see that
\begin{align}
2(i\Theta_{f_2} +r_1 \omega_{\Sigma})(1- \vert \phi \vert_h^2) &= i\Theta_h (2r_2+ \vert \phi \vert_h^2) \nonumber \\
\Rightarrow i\Theta_{f_2} +r_1 \omega_{\Sigma} &= \frac{i\Theta_h (2r_2+ \vert \phi \vert_h^2)}{2(1- \vert \phi \vert_h^2)}.
\label{genma3}
\end{align}
Substituting \ref{genma3} in \ref{genma2} we see that
\begin{gather}
i\Theta_h = (1- \vert \phi \vert_h^2)\frac{\mu\omega_{\Sigma}+i\nabla ^{1,0} \phi \wedge \nabla^{0,1} \phi^{\dag_h}}{(2r_2+ \vert \phi \vert_h^2)(2+2r_2- \vert \phi \vert_h^2)}.
\label{RSPDE}
\end{gather}
By assumption, equation \ref{RSPDE} has a smooth solution satisfying $\vert \phi \vert_h^2 \leq 1$. (Hence $\Theta_h \geq 0$.) We can now solve \ref{genma3} for a smooth $f_2$ and thus we can solve the vbMA equation.
\begin{lemma}
If equation \ref{RSPDE} has a smooth solution $h$ satisfying $\vert \phi \vert_h^2 \leq 1$, then \ref{genma3} has a smooth solution $f_2$ (unique up to a constant).
\label{existenceoff2}
\end{lemma}
\begin{proof}
Using local normal coordinates it is easy to prove the following Weitzenb\"ock-type identity.
\begin{gather}
\pbp \vert \phi \vert_{h}^2 = -\Theta_h \vert \phi \vert_{h}^2 + \nabla^{1,0} \phi \wedge \nabla ^{0,1} \phi ^{\dag_h}. 
\label{verycrucialWB}
\end{gather}
Integrating on both sides we see that
\begin{align}
\displaystyle \int i\Theta_h \vert \phi \vert_h^2 &= \int i\nabla^{1,0} \phi \wedge \nabla ^{0,1} \phi ^{\dag_h} \nonumber \\
&= \displaystyle \int i\Theta \frac{(2r_2+\vert \phi \vert_h^2)(2r_2+2-\vert \phi \vert_h^2)}{1-\vert \phi \vert_h^2}-\int \mu \omega_{\Sigma} \nonumber \\
\Rightarrow \displaystyle \mu  &= \int i\Theta_h \frac{\vert \phi \vert_h^2+2r_2(2r_2+2)}{1-\vert \phi \vert_h^2} \nonumber \\
&= \int i\Theta_h \left [-1+\frac{1+2r_2(2r_2+2)}{1-\vert \phi \vert_h^2} \right ]\nonumber \\
\Rightarrow \frac{\mu+1}{1+2r_2(2r_2+2)} &=\int \frac{i\Theta_h}{1-\vert \phi \vert_h^2}
\label{beforejuncture}
\end{align}
Equation \ref{genma3} is a linear equation and hence has a unique (up to a constant) smooth solution if and only if the right-hand-side is smooth and the integrals on both sides match up. Using the expressions \ref{RSPDE} and \ref{genma3} it is easy to see that the coefficients of the equation are smooth. It remains to check that the integrals are equal.
\begin{align*}
\displaystyle \int (i\Theta_{f_2} +r_1 \omega_{\Sigma})&=r_1 \\
\int \frac{i\Theta_h (2r_2+ \vert \phi \vert_h^2)}{2(1- \vert \phi \vert_h^2)} &= \int \frac{i\Theta_h}{2} \left (\frac{2r_2+1}{1- \vert \phi \vert_h^2}-1\right ) \nonumber \\
&=\frac{2r_2+1}{2}\int \frac{i\Theta_h}{1-\vert \phi \vert_h^2}-\frac{1}{2}.
\end{align*}
At this juncture we may use equation \ref{beforejuncture} to conclude the proof.
\begin{align*}
\int \frac{i\Theta_h (2r_2+ \vert \phi \vert_h^2)}{2(1- \vert \phi \vert_h^2)}&=\frac{2r_2+1}{2}\frac{\mu+1}{1+2r_2(2r_2+2)}-\frac{1}{2} \nonumber \\
&= \frac{(2r_2+1)(2r_1(r_2+1)+2r_2(r_1+1)+1)}{2(2r_2+1)^2}-\frac{1}{2} \\
&= \frac{2r_1(2r_2+1)+2r_2+1}{2(2r_2+1)}-\frac{1}{2} \\
=r_1&= \int (i\Theta_{f_2} +r_1 \omega_{\Sigma}). 
\end{align*}
\end{proof}
\indent Therefore, equation \ref{symmetricvbma} has a smooth solution $H=h_1\oplus g_2$ where $h_1 = \pi_1^{*}(hf_2 h_0^{r_1})\otimes \pi_2^{*}(h_{FS}^{2r_2})$ and $g_2= \pi_1^{*}(f_2 h_0^{r_1})\otimes \pi_2^{*}(h_{FS}^{2r_2+2})$. We just need to prove that $H$ is Griffiths positively curved. Indeed,
\begin{lemma}
$\Theta_H$ is Griffiths positive.
\label{GPH}
\end{lemma}
\begin{proof}
We check the Griffiths positivity of the metric $H$ by looking at the $(1,1)$-form $\vec{v}^{\dag} \Theta \vec{v}$.
\begin{align}
\vec{v}^{\dag} \Theta \vec{v} &= \vert v_1 \vert^2 (\Theta_{h}+\Theta_{f_2}-ir_1 \omega_{\Sigma}-2ir_2 \omega_{FS} -i\vert \phi \vert_h^2 \omega_{FS})\nonumber \\ &+ \vert v_2 \vert ^2 (\Theta_{f_2} -ir_1 \omega_{\Sigma}-(2r_2+2)i \omega_{FS} + i  \vert \phi \vert_h^2 \omega_{FS})\nonumber \\
 &+ \bar{v}_1 v_2 \nabla ^{1,0} \phi \wedge \pi_2 ^{*} \zeta -\bar{v}_2 v_1 \nabla ^{0,1} \phi^{\dag_h} \wedge \pi_2 ^{*} \zeta^{\dag} \nonumber \\
&=  \vert v_1 \vert^2 (\Theta_{h}+\Theta_{f_2}-ir_1 \omega_{\Sigma})+  \vert v_2 \vert ^2 (\Theta_{f_2} -ir_1 \omega_{\Sigma})  \nonumber \\
&+ \omega_{FS} (\vert v_1 \vert ^2 (-2ir_2 -i \vert \phi \vert_h^2) + \vert v_2 \vert^2 (-(2r_2+2)i+ i \vert \phi \vert_h^2)) \nonumber \\
 &+ \bar{v}_1 v_2 \nabla ^{1,0} \phi \wedge \pi_2 ^{*} \zeta -\bar{v}_2 v_1 \nabla ^{0,1} \phi^{\dag_h} \wedge \pi_2 ^{*} \zeta^{\dag}. 
\label{griffpositivityform}
\end{align} 
The form $i \vec{v}^{\dag} \Theta \vec{v}$ is positive for all $\vec{v} \neq 0$ if and only if
\begin{align}
i\Theta_h+i\Theta_{f_2} + r_1 \omega_{\Sigma} &> 0, \nonumber \\
i\Theta_{f_2} + r_1 \omega_{\Sigma} &>0, \nonumber \\
(2r_2+2) -  \vert \phi \vert_h^2 &>0, \nonumber
\end{align}
and
\begin{gather}
0<\frac{(i \vec{v}^{\dag} \Theta \vec{v})^2}{2} = \omega_{FS} \Bigg ( \vert v_1 \vert^4 (2r_2+ \vert \phi \vert_h^2)(i\Theta_h+i\Theta_{f_2}+r_1 \omega_{\Sigma})\nonumber \\
 + \vert v_2 \vert^4 (2r_2+2- \vert \phi \vert_h^2)(i\Theta_{f_2}+r_1 \omega_{\Sigma}) \nonumber \\
+ \vert v_1 v_2 \vert^2 \left(i\Theta_h(2r_2+2- \vert \phi \vert_h^2)+(4r_2+2)(i\Theta_{f_2}+r_1 \omega_{\Sigma})-i \nabla ^{1,0} \phi \wedge \nabla^{0,1} \phi^{\dag_h} \right )\Bigg ).
\label{griffposconditions}
\end{gather}
In our case we may use equations \ref{genma1} and \ref{genma2} to calculate.
\begin{align}
\vert \phi \vert_h^2 &\leq 1 <2r_2+2 \nonumber \\
i\Theta_h+i\Theta_{f_2} + r_1 \omega_{\Sigma} &=\frac{\mu \omega_{\Sigma}+i\nabla^{1,0}\phi \wedge \nabla^{0,1} \phi^{\dag_h}}{2(2r_2+\vert \phi \vert_h^2)}>0\nonumber \\
i\Theta_{f_2} + r_1 \omega_{\Sigma} &= \frac{\mu \omega_{\Sigma}+i\nabla^{1,0}\phi \wedge \nabla^{0,1} \phi^{\dag_h}}{2(2r_2+2-\vert \phi \vert_h^2)}>0.
\label{halfconditions1}
\end{align}
As for the inequality \ref{griffposconditions}, it is clear that the coefficients of $\vert v_1 \vert^4$ and $\vert v_2 \vert^4$ are positive. We will prove that the coefficient of $\vert v_1 v_2 \vert^2$ is non-negative. In what follows, we use expressions \ref{RSPDE} and \ref{halfconditions1} to calculate.
\begin{align}
&\left(i\Theta_h(2r_2+2- \vert \phi \vert_h^2)+(4r_2+2)(i\Theta_{f_2}+r_1 \omega_{\Sigma})-i \nabla ^{1,0} \phi \wedge \nabla^{0,1} \phi^{\dag_h} \right ) \nonumber \\
&\geq i\nabla^{1,0} \phi \wedge \nabla^{0,1} \phi^{\dag_h} \left (\frac{1-\vert \phi \vert_h^2}{2r_2+\vert \phi \vert_h^2}+\frac{4r_2+2}{2(2r_2+2-\vert \phi \vert_h^2)}-1 \right ) \nonumber \\
&=i\nabla^{1,0} \phi \wedge \nabla^{0,1} \phi^{\dag_h} \left (\frac{1-\vert \phi \vert_h^2}{2r_2+\vert \phi \vert_h^2}-\frac{1-\vert \phi \vert_h^2}{2r_2+2-\vert \phi \vert_h^2} \right ) \nonumber \\
 &=\frac{2i\nabla^{1,0} \phi \wedge \nabla^{0,1} \phi^{\dag_h} (1-\vert \phi \vert_h^2)^2}{(2r_2+\vert \phi \vert_h^2)(2r_2+2-\vert \phi \vert_h^2)} \nonumber \\
&\geq 0. \nonumber
\end{align}
Hence, the metric $H$ is Griffiths positively curved as desired.
\end{proof}
Therefore, the vbMA equation has been dimensionally reduced to the Monge-Amp\`ere Vortex equation thus proving Theorem \ref{dimensionalreductionthm}.
\end{proof}

\subsection{Existence}
\indent  Now we set up the method of continuity to solve the Monge-Amp\`ere Vortex equation \ref{MAVortexeq}. To this end, we need to choose a good initial metric. We choose it to be $h_0 A$ where $h_0$ satisfies $i\Theta_0 = \omega_{\Sigma}$ and $\vert \phi \vert_0^2 <\frac{1}{2}$ as before, and $A$ is a small scaling factor. This rescaled metric has the same curvature as the original one. \\
\indent Let $T\subset[0,1]$ be the set of all $t$ such that the following equation has a smooth solution $h_t=h_0 e^{-\psi_t}$. 
\begin{gather}
i\Theta_{h_t} = i(\Theta_{h_0}+\pbp \psi_t)= (1- \vert \phi \vert_{h_t}^2)\frac{\mu u^{1-t} \omega_{\Sigma}+it\nabla_t ^{1,0} \phi \wedge \nabla_t^{0,1} \phi^{\dag_{h_t}}}{(2r_2+t \vert \phi \vert_{h_t}^2)(2+2r_2-t \vert \phi \vert_{h_t}^2)},
\label{methodofcontinuity}
\end{gather}
where $u=\frac{1}{\alpha (1-\vert \phi \vert_0^2)}$ and $$\alpha =\frac{\mu}{(2r_2)(2r_2+2)} = \frac{2(2r_1r_2+r_1+r_2)}{4r_2^2+4r_2}>1,$$ if $r_1>r_2$. As we shall see, it is crucial that $\alpha>1$ (for Lemma \ref{upperbound} to hold).\\
\indent At $t=0$ we have a solution $h_0$, i.e., $\psi_0=0$. If we prove that $T$ is open and closed then we will be done. \\
\indent For future use, we record a useful observation :
\begin{lemma}
$ \vert \phi \vert_{h_t}^2 \leq 1$. Thus $\Theta_{h_t} (x) \geq 0 \ \forall \ t \in T, \ x \in X$.
\label{usefulobs}
\end{lemma}
\begin{proof}
Recall equation \ref{verycrucialWB} :
\begin{gather}
\pbp \vert \phi \vert_{h_t}^2 = -\Theta_t \vert \phi \vert_{h_t}^2 + \nabla_t ^{1,0} \phi \wedge \nabla_t ^{0,1} \phi ^{t*}. 
\label{seconderivativeofmodphisquared}
\end{gather}
At the maximum point $p$ of $g=\vert \phi \vert_{h_t}^2$, we know that $\nabla g(p)=0$ and $i\pbp g(p) \leq 0$. Therefore $\nabla_t \phi (p)=0$, and since $\phi$ is not identically $0$, we see that $\Theta_t(p) \geq 0$. From equation \ref{methodofcontinuity} it is now clear that $1- \vert \phi \vert_{h_t}^2(p)\geq 0$. Hence $g(x)\leq g(p) \leq 1 \ \forall  \ x\in X$ as desired.
\end{proof}
Now we prove openness of the set $T$.
\begin{lemma}
The set $T\subset [0,1]$ such that equation \ref{methodofcontinuity} has a smooth solution is open. In particular, there exists a $\delta>0$ such that $[0,\delta) \subset T$.
\label{openness}
\end{lemma}
\begin{proof}
Assume that $t_0 \in T$. We want to prove that $(t_0-\delta, t_0+\delta) \subset T$ for some $\delta>0$ (depending on $t_0$). To this end, let $C_1$ denote the open subset of the space of $C^{2,\alpha}$ functions $\mathcal{C}$ consisting of functions $\psi$ such that $2r+2- \vert \phi \vert_{h_0 e^{-\psi}}^2 > 0$. Let $C_2$ be the space of $C^{0,\alpha}$ $(1,1)$-forms. Define the metric $h$ by $h=h_0e^{-\psi}$. Consider the map $L : C_1 \times \mathbb{R} \rightarrow C_2$ given by
\begin{gather}
L(\psi,t) = i\Theta_{h} - (1- \vert \phi \vert_{h}^2)\frac{\mu u^{1-t} \omega_{\Sigma}+it \nabla_h ^{1,0} \phi \wedge \nabla_h^{0,1} \phi^{\dag_{h}}}{(2r_2+t \vert \phi \vert_{h}^2)(2+2r_2-t \vert \phi \vert_{h}^2)}  \nonumber \\
=i\Theta_h - (1- \vert \phi \vert_h^2) J,
\end{gather}
where $J= \frac{\mu u^{1-t} \omega_{\Sigma}+it \nabla_h ^{1,0} \phi \wedge \nabla_h^{0,1} \phi^{\dag_{h}}}{(2r_2+t \vert \phi \vert_{h}^2)(2+2r_2-t \vert \phi \vert_{h}^2)} $. $L$ is clearly a smooth map between Banach manifolds and $L(\psi_{t_0},t_0)=0$. If we prove that $D_\psi L_{\psi_{t_0},t_0} : \mathcal{C}  \rightarrow C_2$ is an isomorphism, then by the implicit function theorem of Banach manifolds, $(t_0-\delta,t_0+\delta)\subset T$ for some $\delta>0$ as desired. Denote $2r_2+t \vert \phi \vert_{h}^2$ by $I$ and   $2+2r_2-t \vert \phi \vert_{h}^2$ by $II$. Note that $II-I = 2-2t \vert \phi \vert_h^2$. For the sake of convenience we drop the $t$ subscript from $h_t$ from now onwards.\\
\indent Indeed, the derivative along a $C^{2,\alpha}$ function $w$ is 
\begin{gather}
D_{\psi} L_{\psi,t} (w) = i\pbp w - \vert \phi \vert_h ^2 wJ -(1- \vert \phi \vert_h^2)\frac{i t \delta_w \nabla_h ^{1,0}\phi\wedge \nabla_h ^{0,1} \phi ^{h*}}{(I)(II)}\nonumber \\
- w(1- \vert \phi \vert_h^2)t  \vert \phi \vert_h^2J \left ( \frac{1}{I} -\frac{1}{II}\right )
\label{firstcomp}
\end{gather}

Now we use equation \ref{seconderivativeofmodphisquared} to see that
\begin{align}
\delta_w( \nabla_h ^{1,0}\phi\wedge \nabla_h ^{0,1} \phi ^{h*}) &= \delta_w  \pbp \vert \phi \vert_h^2 + \delta_w (\Theta_h \vert \phi \vert_h^2) \nonumber \\
&=\pbp (-\vert \phi \vert_h^2 w) +\pbp w\vert \phi \vert_h^2 -\Theta_h \vert \phi \vert_h ^2 w. 
\label{variationofgradientterm}
\end{align}
Using \ref{variationofgradientterm} in \ref{firstcomp} we get, 
\begin{gather}
D_\psi L_{\psi,t} (w) =i\pbp w - \vert \phi \vert_h ^2 wJ 
- (1- \vert \phi \vert_h^2)(1- t\vert \phi \vert_h^2) \frac{2wt \vert \phi \vert_h^2 J}{(I)(II)} \nonumber \\
-(1- \vert \phi \vert_h^2)\frac{i t (\pbp (-\vert \phi \vert_h^2 w) +\pbp w\vert \phi \vert_h^2 -\Theta_h \vert \phi \vert_h ^2 w)}{(I)(II)}.
\end{gather} 
By the Fredholm alternative, if we prove that the kernel of $DL^{*}(v)$ is trivial, then we will be done. Indeed, a calculation shows that
\begin{align}
(D_\psi L_{\psi,t})^{*} (v)&=i\pbp v- \vert \phi \vert_h ^2 vJ 
- (1- \vert \phi \vert_h^2)(1- t\vert \phi \vert_h^2) \frac{2vt \vert \phi \vert_h^2 J}{(I)(II)}+(1- \vert \phi \vert_h^2)\frac{it \Theta_h \vert \phi \vert_h^2 v}{(I)(II)} \nonumber \\
&-it  \Bigg (\vert \phi \vert_h^2\pbp \left (\frac{-(1- \vert \phi \vert_h^2)v}{(I)(II)} \right )  +\pbp \Big (\frac{v\vert \phi \vert_h^2 (1- \vert \phi \vert_h^2)}{(I)(II)} \Big ) \Bigg ) \nonumber \\
&=i\pbp v- \vert \phi \vert_h ^2 vJ 
- (1-\vert \phi \vert_h^2)(1- t\vert \phi \vert_h^2) \frac{2vt \vert \phi \vert_h^2 J}{(I)(II)}+(1- \vert \phi \vert_h^2)\frac{it \Theta_h \vert \phi \vert_h^2 v}{(I)(II)} \nonumber \\
&-i t \Bigg (\frac{v (1- \vert \phi \vert_h^2)}{(I)(II)}  \pbp \vert \phi \vert_h^2 + \partial \left ( \frac{v(1- \vert \phi \vert_h^2)}{(I)(II)} \right ) \wedge \bar{\partial} \vert \phi \vert_h^2 + \partial \vert \phi \vert_h^2 \wedge \bar{\partial} \left ( \frac{v(1- \vert \phi \vert_h^2)}{(I)(II)} \right )\Bigg ).
\label{adjointmore}
\end{align} 
Using \ref{seconderivativeofmodphisquared} and the fact that $i\Theta_h = (1- \vert \phi \vert_h^2) J$ we see that
\begin{align}
(D_\psi L_{\psi,t})^{*} (v) &= i\pbp v- \vert \phi \vert_h ^2 vJ -(1- \vert \phi \vert_h^2)(1-t) \frac{2vt \vert \phi \vert_h^4 J}{(I)(II)}
-(1- \vert \phi \vert_h^2)\frac{it \nabla^{1,0} \phi \wedge \nabla^{0,1} \phi^{*_h} v}{(I)(II)} \nonumber \\
&-it  \Bigg (\partial \left ( \frac{v(1- \vert \phi \vert_h^2)}{(I)(II)} \right ) \wedge \bar{\partial} \vert \phi \vert_h^2 + \partial \vert \phi \vert_h^2 \wedge \bar{\partial} \left ( \frac{v(1- \vert \phi \vert_h^2)}{(I)(II)} \right )\Bigg ) \nonumber \\
&= i\pbp v- \vert \phi \vert_h ^2 vJ -(1- \vert \phi \vert_h^2)(1-t) \frac{2vt \vert \phi \vert_h^4 J}{(I)(II)}
-(1- \vert \phi \vert_h^2)\frac{it \nabla^{1,0} \phi \wedge \nabla^{0,1} \phi^{*_h} v}{(I)(II)} \nonumber \\
&-it  \Bigg (\frac{1- \vert \phi \vert_h^2}{(I)(II)}\partial v\wedge \bar{\partial}\vert \phi \vert_h^2 +v\partial \left ( \frac{1- \vert \phi \vert_h^2}{(I)(II)} \right ) \wedge \bar{\partial} \vert \phi \vert_h^2 \nonumber \\
 &+ v\partial \vert \phi \vert_h^2 \wedge \bar{\partial} \left ( \frac{1- \vert \phi \vert_h^2}{(I)(II)} \right )+ \frac{1- \vert \phi \vert_h^2}{(I)(II)}\partial \vert \phi \vert_h^2 \wedge \bar{\partial} v \Bigg ).
\label{adjointmoremore}
\end{align}
Now notice that 
\begin{align}
\partial \left ( \frac{1- \vert \phi \vert_h^2}{(I)(II)} \right ) &= -\frac{1}{(I)(II)}\partial \vert \phi \vert_h^2 +\frac{1- \vert \phi \vert_h^2}{(I)(II)} \left (-\frac{\partial I}{I}-\frac{\partial (II)}{II}\right)  \nonumber \\
&=-\frac{1}{(I)(II)}\partial \vert \phi \vert_h^2 -\frac{2(1- \vert \phi \vert_h^2)(1- t\vert \phi \vert_h^2) \partial \vert \phi \vert_h^2}{(I)^2(II)^2}.
\label{anintermediatestep}
\end{align}
Using this in \ref{adjointmoremore} we see that for $v\in ker(DL^{*})$,
\begin{align}
0 &= i\pbp v- \vert \phi \vert_h ^2 vJ  -(1- \vert \phi \vert_h^2)(1- t) \frac{2vt \vert \phi \vert_h^4 J}{(I)(II)}
-(1- \vert \phi \vert_h^2)\frac{it \nabla^{1,0} \phi \wedge \nabla^{0,1} \phi^{*_h} v}{(I)(II)} \nonumber \\
&-it  \Bigg (\frac{1- \vert \phi \vert_h^2}{(I)(II)}\partial v\wedge \bar{\partial}\vert \phi \vert_h^2 + \frac{1- \vert \phi \vert_h^2}{(I)(II)}\partial \vert \phi \vert_h^2 \wedge \bar{\partial} v  \nonumber \\
&-2 \frac{ v\partial \vert \phi \vert_h^2 \wedge \bar{\partial} \vert \phi \vert_h^2}{(I)(II)} \left (1+\frac{2(1- \vert \phi \vert_h^2)(1- t\vert \phi \vert_h^2)}{(I)(II)} \right ) \Bigg ).
\label{adjointevenmore}
\end{align}
\indent Let $\tilde{v}=v I$. We calculate the derivatives of $\tilde{v}$ now.
\begin{gather}
\bar{\partial} \tilde{v} = tv\bar{\partial}\vert \phi \vert_h^2 +I \bar{\partial} v \nonumber \\
\Rightarrow \partial \bar{\partial} \tilde{v} =  t \partial v \wedge \bar{\partial} \vert \phi \vert_h^2+tv \partial \bar{\partial} \vert \phi \vert_h^2 + t\partial \vert \phi \vert_h^2 \wedge \bar{\partial}v + I \partial \bar{\partial} v.
\label{derivativesoftildev}
\end{gather}
 At the point $p$ where the maximum of $\tilde{v}$ occurs, $i\pbp \tilde{v} (p)\leq 0$ and $\partial \tilde{v}(p)=\bar{\partial} \tilde{v}(p)=0$. The latter observation implies that 
\begin{gather}
\bar{\partial} v(p) =-\frac{tv\bar{\partial}\vert \phi \vert_h^2}{I} 
\label{atthemaxderi}
\end{gather} 
Hence,
\begin{gather}
0\geq -2i\frac{t^2v\partial\vert \phi \vert_h^2 \wedge \bar{\partial}\vert \phi \vert_h^2}{I}(p)-tvi\Theta_h \vert \phi \vert_h^2(p)+itv\nabla^{1,0} \phi \wedge \nabla^{0,1} \phi^{\dag_h} (p) + Ii\partial \bar{\partial} v(p) \nonumber \\
\Rightarrow 0 \geq -2i\frac{t^2v\vert \phi \vert_h^2 \nabla^{1,0} \phi  \wedge \nabla^{0,1} \phi ^{\dag_h}}{I}(p)-tvi\Theta_h \vert \phi \vert_h^2(p)+itv\nabla^{1,0} \phi \wedge \nabla^{0,1} \phi^{\dag_h} (p) + Ii\partial \bar{\partial} v(p)
\label{atthemax}
\end{gather}
We suppress the point $p$ in what follows. Substituting \ref{adjointevenmore} in the above inequality and dividing both sides by $I$ we get the following.
\begin{align*}
0&\geq -2i\frac{t^2v\vert \phi \vert_h^2 \nabla^{1,0} \phi  \wedge \nabla^{0,1} \phi ^{\dag_h}}{I^2}-tv\frac{(1-\vert \phi \vert_h^2)J \vert \phi \vert_h^2}{I}+\frac{itv}{I}\nabla^{1,0} \phi \wedge \nabla^{0,1} \phi^{\dag_h} \nonumber \\
&+ \vert \phi \vert_h ^2 vJ  +(1- \vert \phi \vert_h^2)(1- t) \frac{2vt \vert \phi \vert_h^4 J}{(I)(II)}
+(1- \vert \phi \vert_h^2)\frac{it \nabla^{1,0} \phi \wedge \nabla^{0,1} \phi^{*_h} v}{(I)(II)} \nonumber \\
&+it  \Bigg (\frac{1- \vert \phi \vert_h^2}{(I)(II)}\partial v\wedge \bar{\partial}\vert \phi \vert_h^2 + \frac{1- \vert \phi \vert_h^2}{(I)(II)}\partial \vert \phi \vert_h^2 \wedge \bar{\partial} v  \nonumber \\
 &-2 \frac{ v\partial \vert \phi \vert_h^2 \wedge \bar{\partial} \vert \phi \vert_h^2}{(I)(II)} \left (1+\frac{2(1- \vert \phi \vert_h^2)(1- t\vert \phi \vert_h^2)}{(I)(II)} \right ) \Bigg ).
\end{align*}
Using \ref{atthemaxderi} in the above equation we may simplify it.
\begin{align}
0&\geq -2i\frac{t^2v\vert \phi \vert_h^2 \nabla^{1,0} \phi  \wedge \nabla^{0,1} \phi ^{\dag_h}}{I^2}-tv\frac{(1-\vert \phi \vert_h^2)J \vert \phi \vert_h^2}{I}+i\frac{tv}{I}\nabla^{1,0} \phi \wedge \nabla^{0,1} \phi^{\dag_h} \nonumber \\
&+ \vert \phi \vert_h ^2 vJ  +(1- \vert \phi \vert_h^2)(1- t) \frac{2vt \vert \phi \vert_h^4 J}{(I)(II)}
+(1- \vert \phi \vert_h^2)\frac{it \nabla^{1,0} \phi \wedge \nabla^{0,1} \phi^{*_h} v}{(I)(II)} \nonumber \\
&-\frac{itv}{I}  \Bigg (2t\frac{1- \vert \phi \vert_h^2}{(I)(II)}\vert \phi \vert_h^2 \nabla^{1,0} \phi \wedge \nabla^{0,1}\phi^{\dag_h} \nonumber\\
 &+2 \frac{ \vert \phi \vert_h^2 \nabla^{1,0} \phi \wedge \nabla^{0,1}\phi^{\dag_h}}{II} \left (1+\frac{2(1- \vert \phi \vert_h^2)(1- t\vert \phi \vert_h^2)}{(I)(II)} \right ) \Bigg ) \nonumber \\
&= A i\frac{tv}{I}\nabla^{1,0} \phi \wedge \nabla^{0,1}\phi^{\dag_h} + B\frac{\mu u^{1-t}\omega_{\Sigma}}{(I)(II)}\vert \phi \vert_h^2 v,
\label{semifinalineq}
\end{align}
where
\begin{align*}
A &=-\frac{2t\vert \phi \vert_h^2}{I}-\frac{t(1-\vert\phi \vert_h^2)\vert\phi \vert_h^2}{(I)(II)}+1+\frac{\vert \phi \vert_h^2}{II}+\frac{2(1-\vert \phi \vert_h^2)(1-t)t\vert \phi \vert_h^4}{(I)(II)^2}+\frac{1-\vert \phi \vert_h^2}{II}\nonumber \\
&-\frac{2t(1-\vert \phi \vert_h^2)\vert \phi \vert_h^2}{(I)(II)}-\frac{2\vert \phi \vert_h^2}{II} \left (1+\frac{2(1- \vert \phi \vert_h^2)(1- t\vert \phi \vert_h^2)}{(I)(II)} \right )  \\
\Rightarrow A &= -\frac{2t\vert \phi \vert_h^2}{I}-\frac{t(1-\vert\phi \vert_h^2)\vert\phi \vert_h^2}{(I)(II)}+1+\frac{2(1-\vert \phi \vert_h^2)(1-t)t\vert \phi \vert_h^4}{(I)(II)^2}-\frac{2t(1-\vert \phi \vert_h^2)\vert \phi \vert_h^2}{(I)(II)}\nonumber \\
&+\frac{1}{II} \left (1-2\vert \phi \vert_h^2-\frac{4\vert \phi \vert_h^2(1- \vert \phi \vert_h^2)(1- t\vert \phi \vert_h^2)}{(I)(II)} \right )\\
&\geq -\frac{2}{I}-\frac{1}{4(I)(II)}+1-\frac{1}{2(I)(II)}+\frac{1}{II}\left (-1-\frac{1}{(I)(II)} \right )\\
&\geq -\frac{1}{2}-\frac{1}{80}+1-\frac{1}{40}-\frac{1}{5} \left ( 1+ \frac{1}{20}\right ) >0,
\end{align*}
and \\
\begin{align*}
B &= -\frac{t(1-\vert \phi \vert_h^2)}{I} +1+\frac{2(1-\vert \phi \vert_h^2)(1-t)t\vert \phi \vert_h^2}{(I)(II)} \\
&\geq -\frac{1}{4}+1>0.
\end{align*}
Going back to \ref{semifinalineq} we see that $v(p)\leq 0$. (Indeed, since $L$ has degree $1$, either $\vert \nabla \phi \vert (p) \neq 0$ or $\phi (p) \neq 0$.) Thus, $\tilde{v}_{max} \leq 0$. The same argument shows that $\tilde{v}_{min} \geq 0$. Hence $\tilde{v} \equiv 0 \equiv v$, thus proving that $ker(D_{\psi}L_{\psi,t}^{*})=\{0\}$. Therefore $T$ is open.
\end{proof}

We prove the closedness of $T$ assuming openness near $t=0$, i.e., $[0,\delta) \subset T$ for some $\delta>0$. Suppose $t_n \rightarrow t$ is a sequence such that $h_{t_n}=h_0 e^{-\psi_{t_n}}$ solves \ref{methodofcontinuity}. We need to show that a subsequence $\psi_{t_{n_k}}\rightarrow \psi_t$ in $C^{2,\alpha}$ and that $\psi_t$ is smooth. If we prove that $\Vert \psi_t \Vert_{C^{2,\beta}} \leq C$ where $C$ is independent of $t$, then by the Arzela-Ascoli theorem, for $\alpha<\beta$ we have a convergent subsequence. Lemma \ref{usefulobs} together with elliptic regularity shows that $\psi_t$ is smooth. So we just need to prove $C^{2,\beta}$ \emph{a priori} estimates on $\psi_t$. From now onwards we suppress the dependence of $\psi_t$ on $t$. The following lemma reduces the estimates to a $C^1$ estimate.
\begin{lemma}
Suppose $\Vert \psi \Vert_{C^1}\leq C$. Then $\Vert \psi \Vert_{C^{2,\beta}}\leq C$.
\end{lemma}
\begin{proof}
Under the hypotheses, using Lemma \ref{usefulobs} it is clear that the right hand side of \ref{methodofcontinuity} is uniformly bounded in $C^0$. Therefore, by the $L^p$ regularity of elliptic equations, $\psi$ is bounded uniformly in $W^{2,p}$ for all large $p$. Using the Sobolev embedding theorem we see that $\Vert \psi \Vert_{C^{1,\beta}}\leq C$. Thus the right hand side is in $C^{0,\beta}$. By the Schauder estimates we are done.
\end{proof}

Now we reduce the $C^1$ estimate to a uniform estimate on $\psi$.
\begin{lemma}
If $\Vert \psi \Vert_{C^0} \leq C$, then $\Vert \psi \Vert_{C^1(X)} \leq \tilde{C}$.
\label{desiredconeestimate}
\end{lemma}
\begin{proof}
 To produce a contradiction, assume that there exists a sequence $\psi_n$ (corresponding to $t_n \rightarrow t$) such that $\max_X \vert d\psi_n \vert =\vert d \psi_n(p_n) \vert=M_n \rightarrow \infty$. Up to a subsequence, we may assume that $p_n \rightarrow p$. Choose $n$ large enough so that $p_n, p$ lie in a coordinate ball $\tilde{B}$ centred at $p$ with coordinates $z$ (with $z=0$ corresponding to $p$). Define $\tilde{\psi}_n (\tilde{z}) = \psi(p_n+\frac{\tilde{z}}{M_n})$. Now $\vert \tilde{d} \tilde{\psi}_n \vert \leq 1=\vert \tilde{d} \tilde{\psi}_n \vert(0)$. Note that 
\begin{gather}
\frac{\partial\tilde{\psi}_n}{\partial \tilde{z}} = \frac{1 }{M_n}\frac{\partial \psi}{\partial z}, \frac{\partial\tilde{\psi}_n}{\partial \overline{\tilde{z}}} = \frac{1 }{M_n}\frac{\partial \psi}{\partial \bar{z}} \nonumber \\
\frac{\partial^2\tilde{\psi}_n}{\partial \tilde{z}\partial \overline{\tilde{z}}} = \frac{1 }{M_n^2}\frac{\partial \psi}{\partial z \partial \bar{z}}.
\label{smallcalculations}
\end{gather}
 We abuse notation from this point onwards and denote the functions $\frac{\omega_{\Sigma}}{idz\wedge d\bar{z}}= \frac{\Theta_0}{dz\wedge d\bar{z}}$ by $\omega_{\Sigma}$ and $\frac{\nabla ^{1,0} \phi \wedge \nabla ^{0,1} \phi^{*}}{d\tilde{z}\wedge d\overline{\tilde{z}}}$ by $\tilde{\nabla} ^{1,0} \phi \wedge \tilde{\nabla} ^{0,1} \phi^{*}$. Using \ref{smallcalculations} and \ref{methodofcontinuity} we see that 
\begin{gather}
\frac{\omega_{\Sigma}}{M_n^2} + i \frac{\partial^2 \tilde{\psi}_n}{\partial \tilde{z} \partial \overline {\tilde{z}}} 
=  \frac{1- \vert \phi \vert_n^2 }{(I_n)(II_n)}\Bigg [ i t\tilde{\nabla}_n^{1,0} \phi \wedge \tilde{\nabla}_n^{0,1} \phi^{n*} +\frac{\mu \omega_{\Sigma}u^{1-t_n}}{M_n^2}\Bigg ]. 
\label{rescaledcalc}
\end{gather}
We observe that since $ \Vert \phi \Vert_n^2 \leq 1$, the denominator in \ref{rescaledcalc} is bounded below.\\
\indent On a coordinate ball $B_R(0)$ in the $\tilde{z}$ coordinates we have $\vert \tilde{d} \tilde{\psi}_n \vert \leq 1$. Using \ref{rescaledcalc} we conclude that $\vert \tilde{\Delta} \tilde{\psi}_n \vert \leq C$ on $B_R(0)$. Hence, by interior $L^p$ regularity and Sobolev embedding we see that $\Vert \tilde{\psi}_n \Vert_{C^{1,\beta}(0.6B_R(0))} \leq C$. Thus by the interior Schauder estimates $\Vert \tilde{\psi}_n \Vert_{C^{2,\beta}(0.5B_R(0))} \leq C$. Suppose $\Vert \tilde{\psi}_n \Vert_{C^{2,\beta}(B_{0.5R}(0))} \leq C_R$ for some fixed $\beta>0$. For every fixed $R$, a subsequence of $\tilde{\psi}_n$ converges in $C^{2,\alpha}(B_{0.5R}(0))$ to a function $\tilde{\psi}_R$ for a fixed $\alpha <\beta$. Choosing a diagonal subsequence we may assume that for all $R$ we have a single function $\tilde{\psi}$. It is easy to see using \ref{rescaledcalc} that $i\frac{\partial^2 \tilde{\psi}}{\partial \tilde{z}\partial \overline{\tilde{z}}} \geq 0$. But a subharmonic function on $\mathbb{C}$ cannot be bounded above unless it is a constant. Hence $\tilde{\psi}$ is a constant. But this contradicts the fact that $\vert \tilde{d} \tilde{\psi} \vert(0)=1$.  Hence $\vert \nabla_0 \psi \vert \leq C$ thus implying a $C^1$ estimate. 
\end{proof}
\indent For the $C^0$ estimate we need the following form of Green's representation formula. Let $G$ be a Green's function of the fixed background metric $\omega_{\Sigma}$ such that $-C[1+\vert\ln(d_{\omega_{\Sigma}}(P,Q))\vert]\leq G(P,Q) \leq 0$. Then any continuous function $w$ satisfies the following equation.
\begin{gather}
w(Q)=\displaystyle \frac{\int w\omega_{\Sigma} }{\int \omega_{\Sigma}} + \int_{\Sigma} G(P,Q) i\pbp w(Q). 
\label{Green}
\end{gather}
 Next we prove a lower bound on $\psi$ :
\begin{lemma}
The function $\psi$ satisfies $\psi \geq -C$, where $C$ is independent of $t$. \label{lowerbound}
\end{lemma}
\begin{proof}
Firstly, note that since $\vert \phi \vert_{0} ^2 e^{-\psi} \leq 1$, 
\begin{gather}
\displaystyle \int \psi \omega_{\Sigma} \geq \int \ln (\vert \phi \vert_{h_0}^2) \omega_{\Sigma}.
\end{gather}
Using Green's formula \ref{Green} we see that
\begin{align}
\psi (P) &\geq   \displaystyle \int \ln (\vert \phi \vert_{h_0}^2) \omega_{\Sigma} + \displaystyle \int G(P,Q)i\pbp \psi(Q) \nonumber \\
&= \int \ln (\vert \phi \vert_{h_0}^2) \omega_{\Sigma} + \displaystyle \int G(P,Q)i\Theta_h(Q) - \int G(P,Q) \omega_{\Sigma}(Q).
\label{beforethetaingreen}
\end{align}
At this point we note that
\begin{align}
i\Theta_h &\leq \frac{\mu u^{1-t}\omega_{\Sigma}+i\pbp \vert \phi \vert_h^2+i\Theta_h \vert \phi \vert_h^2}{2r_2(2r_2+1)} \nonumber \\
\Rightarrow i\Theta_h (2r_2(2r_2+1)-1)&\leq i\Theta_h  (2r_2(2r_2+1)-\vert \phi \vert_h^2)\nonumber \\
&\leq \mu u^{1-t}+i\pbp \vert \phi \vert_h^2 \nonumber \\
\Rightarrow G(P,Q) i\Theta_h (Q) &\geq G(P,Q)\frac{ \mu u^{1-t}(Q) \omega_{\Sigma}(Q)+i\pbp \vert \phi \vert_h^2(Q)}{2r_2(2r_2+1)-1}. 
\label{thetahbounds}
\end{align}
Using \ref{thetahbounds} in \ref{beforethetaingreen} we get
\begin{align}
\psi(P) &\geq \int \ln (\vert \phi \vert_{h_0}^2) \omega_{\Sigma} + \displaystyle \int  G(P,Q)\frac{ \mu u^{1-t}(Q)\omega_{\Sigma}(Q)+i\pbp \vert \phi \vert_h^2(Q)}{2r_2(2r_2+1)-1}- \int G(P,Q) \omega_{\Sigma}(Q) \nonumber \\
&= \ln (\vert \phi \vert_{h_0}^2) \omega_{\Sigma} + \displaystyle \int  G(P,Q)\frac{ \mu u^{1-t}(Q)\omega_{\Sigma}(Q)}{2r_2(2r_2+1)-1}- \int G(P,Q) \omega_{\Sigma}(Q) + \frac{\vert \phi \vert_h^2 (P)-\displaystyle \int \vert \phi \vert_h^2 \omega_{\Sigma}}{2r_2(2r_2+1)-1} \nonumber \\
&\geq \ln (\vert \phi \vert_{h_0}^2) \omega_{\Sigma} + \displaystyle \int  G(P,Q)\frac{ \mu u^{1-t}(Q)}{2r_2(2r_2+1)-1}\omega_{\Sigma}(Q)- \int G(P,Q) \omega_{\Sigma}(Q) - \frac{1}{2r_2(2r_2+1)-1} \nonumber \\
&\geq -C, \nonumber 
\end{align}
where we used the Green representation formula again in the third-to-last inequality.
\end{proof}
\indent Finally we prove an upper bound on $\psi$ thus proving the closedness of $T$.
\begin{lemma}
The function $\psi$ satisfies $\psi \leq C,$ where $C$ is independent of $t$.
\label{upperbound}
\end{lemma}
\begin{proof}
Suppose the maximum of $\psi$ is attained at a point $p$. At this point, we see that $i\pbp \psi (p) \leq 0$ and $\partial \psi(p)=\bar{\partial} \psi(p)=0$. Thus
\begin{gather}
i\Theta_0 (p)=\omega_{\Sigma}(p)\geq (1-\vert \phi \vert_{h}^2 (p))\frac{\mu u(p)^{1-t}\omega_{\Sigma}(p)}{(I)(p)(II)(p)}. \label{atp}
\end{gather}
If the upper bound does not hold and a sequence of $\psi_n (p_n) \rightarrow \infty$ (with $p_n \rightarrow q$), then $\vert \phi \vert_n ^2 (q) \rightarrow 0$. Hence,
\begin{gather}
1\geq \frac{\mu u^{1-t}(q)}{2r_2(2r_2+2)} = \frac{\alpha^{t}}{ (1-\vert \phi \vert_{h_0}^2(q))^{1-t}} \geq \alpha^t >\alpha^{\delta/2}>1.
\end{gather}
 We have a contradiction. Hence $\psi \leq C$.
\end{proof}
Therefore $T$ is open, closed, and non-empty. Hence $T=[0,1]$ as desired. \\

\textbf{Postscript} : A slightly unsatisfactory aspect of the above proof is that the method of continuity used above is tricky to generalise to an arbitrary vector bundle. The usual method of continuity used for the complex Monge-Amp\`ere equation is $$\omega_{\phi_t}^n = e^{tf} \omega^n.$$
A naive generalisation of this to the Monge-Amp\`ere Vortex equation \ref{MAVortexeq} seems to cause problems with regard to the upper bound on $\psi$, i.e., Lemma \ref{upperbound}. That being said, it is quite possible that an involved Moser iteration argument (akin to the one used for Calabi Conjecture) might circumvent this problem. 

\subsection{$r_1>r_2$ assuming existence}
\indent  We now show that if a solution to equation \ref{MAVortexeq} exists, then $r_1>r_2$. Indeed, if a solution exists then it also solves the vbMA equation by Theorem \ref{dimensionalreductionthm}. Theorem \ref{stabilityandChernclassinequality} shows that the bundle is MA-stable. This means that 
$$\mu_{MA} (\pi_1^{*}(r_1+1)L\otimes \pi_2^{*} r_2 \mathcal{O}(2)) < \mu_{MA} (V).$$
We now calculate the second Chern characters of the bundles involved.
\begin{gather}
\mathrm{ch}_2 (\pi_1^{*}(r_1+1)L\otimes \pi_2^{*} r_2 \mathcal{O}(2)) = \frac{\mathrm{ch}_1 (\pi_1^{*}(r_1+1)L)) \mathrm{ch}_1(\pi_2^{*} r_2 \mathcal{O}(2))}{2} \nonumber \\
= (r_1+1)r_2 c_1(L) [\mathbb{P}^1] \nonumber \\
\mathrm{ch}_2 (V) =  \mathrm{ch}_2 (\pi_1^{*}(r_1+1)L\otimes \pi_2^{*} r_2 \mathcal{O}(2)) + \mathrm{ch}_2 (\pi_1^{*}r_1L\otimes \pi_2^{*} (r_2+1) \mathcal{O}(2))  \nonumber \\
= ((r_1+1)r_2  + r_1(r_2+1))c_1(L) [\mathbb{P}^1].
\end{gather}
Therefore the stability condition translates into the following inequality.
\begin{gather}
(r_1+1)r_2 < \frac{(r_1+1)r_2  + r_1(r_2+1)}{2} \nonumber \\
\Rightarrow r_2<r_1.
\end{gather}
\begin{remark}
Interestingly enough, it can be shown that $V$ is Mumford stable with respect to $c_1(V)$. Indeed, by a result of Garc\'ia-Prada \cite{Oscar} it suffices to check this for the subbundle $S=\pi_1^{*}(r_1+1)L\otimes \pi_2^{*} r_2 \mathcal{O}(2)$.
\begin{gather}
c_1(\pi_1^{*}(r_1+1)L\otimes \pi_2^{*} r_2 \mathcal{O}(2)).c_1(V) = ((r_1+1)c_1(L)+2r_2 [\mathbb{P}^1]).((2r_1+1)c_1(L)+2(2r_2+1) [\mathbb{P}^1]) \nonumber \\
= (2(r_1+1)(2r_2+1)+2r_2(2r_1+1))c_1 (L). [\mathbb{P}^1]\nonumber \\
c_1(V).c_1(V) = 4(2r_1+1)(2r_2+1) \nonumber \\
\frac{deg(S)}{1}-\frac{deg(V)}{2}=2(r_1+1)(2r_2+1)+2r_2(2r_1+1)-2(2r_1+1)(2r_2+1) \nonumber \\
= -2r_1+2r_2<0.
\end{gather}
As remarked in the introduction, this explains why a KLBMY-type inequality holds in this case.
\end{remark}

\subsection{Uniqueness}
\indent  To complete the proof of Theorem \ref{KobHitcorrespondence} we need to prove uniqueness. Our strategy to do so is as follows :
\begin{enumerate}
\item Let $\mathfrak{h}_1$ denote the $t=1$ solution arising from the continuity path \ref{methodofcontinuity}, i.e.,
\begin{gather}
i\Theta_{h_t} = i(\Theta_{h_0}+\pbp \psi_t)= (1- \vert \phi \vert_{h_t}^2)\frac{\mu u^{1-t} \omega_{\Sigma}+it\nabla_t ^{1,0} \phi \wedge \nabla_t^{0,1} \phi^{\dag_{h_t}}}{(2r_2+t \vert \phi \vert_{h_t}^2)(2+2r_2-t \vert \phi \vert_{h_t}^2)}.
\label{methodofcontinuityagain}
\end{gather}
\indent Let $h_2$ denote any other solution of the Monge-Amp\`ere Vortex equation satsifying $\vert \phi \vert_{h_2}^2 \leq 1$. We wish to run another continuity path backwards starting with $\tilde{h}_{t=1} = h_2 : $ 
\begin{gather}
i\Theta_{\tilde{h}_t} = i(\Theta_{h_0}+\pbp \tilde{\psi}_t)= (1- \vert \phi \vert_{\tilde{h}_t}^2)\frac{\mu u^{1-t} \omega_{\Sigma}+it\nabla_{\tilde{h}_t} ^{1,0} \phi \wedge \nabla ^{0,1} \phi^{\dag_{\tilde{h}_t}}}{(2r_2+t \vert \phi \vert_{\tilde{h}_t}^2)(2+2r_2-t \vert \phi \vert_{\tilde{h}_t}^2)}.
\label{backwardspath}
\end{gather}
Denote by $\tilde{T}\subset [0,1]$ the set of $t$ such that \ref{backwardspath} has a smooth solution. This set is non-empty because $1 \in \tilde{T}$.
\item The proof of openness for \ref{methodofcontinuityagain} shows that the set $\tilde{T}\subset [0,1]$ is open. The \emph{a priori} estimates for \ref{methodofcontinuityagain} show that $(0,1] \subset \tilde{T}$. The only potential problem can occur at $t=0$ because Lemma \ref{upperbound} is no longer valid. 
\item We prove that there exists a ``small" $t_0 \in [0,1]$ such that \ref{methodofcontinuityagain} has a unique solution at $t=t_0$. That is, there exists a unique smooth $h$ satisfying $\vert \phi \vert_h^2 \leq 1$ and
\begin{gather}
i\Theta_{h} = i(\Theta_{h_0}+\pbp \psi)= (1- \vert \phi \vert_{h}^2)\frac{\mu u^{1-t_0} \omega_{\Sigma}+it_0\nabla_h ^{1,0} \phi \wedge \nabla^{0,1} \phi^{\dag_{h}}}{(2r_2+t_0 \vert \phi \vert_{h}^2)(2+2r_2-t_0 \vert \phi \vert_{h}^2)}.
\label{uniquenessforsmallt0strategyeq}
\end{gather}
\item The previous point implies that the two continuity paths \ref{methodofcontinuityagain} and \ref{backwardspath} intersect at $t=t_0$. Using this observation we prove that $h_1=h_2$.
\end{enumerate}
\indent Since steps 1 and 2 are already done, we proceed to step 3 :
\begin{lemma}
There exists a number $t_0 \in (0,1]$ depending only on $r_2, \mu, h_0$ such that there is a unique smooth metric $h$ satisfying $\vert \phi \vert_h^2 \leq 1$ and the following equation.
\begin{gather}
i\Theta_{h} = i(\Theta_{h_0}+\pbp \psi)= (1- \vert \phi \vert_{h}^2)\frac{\mu u^{1-t_0} \omega_{\Sigma}+it_0\nabla_h ^{1,0} \phi \wedge \nabla^{0,1} \phi^{\dag_{h}}}{(2r_2+t_0 \vert \phi \vert_{h}^2)(2+2r_2-t_0 \vert \phi \vert_{h}^2)}.
\label{uniqenessforsmallt0eq}
\end{gather}
\label{uniquenessforsmallt0lem}
\end{lemma}
\begin{proof}
As we go along the proof, we will choose $t_0$ to be a small enough number depending only on $r_2, \mu, h_0$.\\
Let $\hon$ be the solution $h_{t_0}$ coming from the continuity path \ref{methodofcontinuityagain}. Denote by $\htw$ any other smooth solution of \ref{uniqenessforsmallt0eq} satisfying $\vert \phi \vert_{\htw}^2 \leq 1$. Define a function $f$ to satisfy $\htw = \hon e^{-f}$. Without loss of generality, there exists a point $q$ such that
\begin{gather}
f(q)>0.
\label{assumption:onf}
\end{gather}
 Let $\h_s = \h_1 e^{-sf}=\h_2^s \h_1^{1-s}$ where $0\leq s\leq 1$. It is easy to see that 
\begin{gather}
\vert \phi \vert_{\h_s}^2 =(\vert \phi \vert_{\h_1})^{2(1-s)} \vert\phi \vert_{\h_2} ^{2s} \leq 1.
\label{boundonphi}
\end{gather}
Let $I_s = 2r_2+t_0\vert \phi \vert_{\h_s}^2$, $II_s = 2r_2+2-t_0\vert \phi \vert_{\h_s}^2$, and $J_s = \frac{\mu u^{1-t_0}\omega_{\Sigma}+it_0 \nabla_{\h_s} ^{1,0} \phi \wedge \nabla^{0,1} \phi^{\dag_{\h_s}}}{(I)_s (II)_s}$. Therefore $i\Theta_{\h_s} = (1-\vert \phi \vert_s^2)J_s$. \\
\indent By assumption, $i\Theta_{\h_1} = (1-\vert \phi \vert_{\h_1})^2 J_1$ and $i\Theta_{\h_2} = (1-\vert \phi \vert_{\h_2})^2 J_2$. Upon subtraction we get the following equation.
\begin{align}
i\partial \bar{\partial}f &= \displaystyle \int_0 ^1 ds \frac{d((1-\vert \phi \vert_{\h_s}^2)J_s)}{ds} \nonumber \\
&= \displaystyle \int_0 ^1 ds \left ( f\vert\phi\vert_{\h_s}^2 J_s + (1-\vert \phi \vert_{\h_s}^2) \frac{dJ_s}{ds}\right ).
\label{FTC}
\end{align}
We now calculate $\frac{dJ_s}{ds}$.
\begin{align}
\frac{dJ_s}{ds} &= -J_s\left ( \frac{1}{I_s}\frac{d (I)_s}{ds} + \frac{1}{II_s}\frac{d (II)_s}{ds}\right ) + \frac{it_0 \frac{d}{ds}\nabla_{\h_s}^{1,0}\phi \wedge \nabla^{0,1} \phi^{\dag_{\h_s}}}{(I)_s(II)_s} \nonumber \\
&= \frac{2J_s t_0 f \vert \phi \vert_{\h_s}^2 (1-t_0 \vert \phi \vert_{\h_s}^2)}{(I)_s (II)_s}+\frac{it_0}{(I)_s(II)s}\frac{d}{ds} \left (\partial \bar{\partial} \vert \phi \vert_{\h_s}^2+\Theta_{\h_s} \vert \phi \vert_{\h_s}^2 \right ),
\label{derivjsfirst}
\end{align}
where we used the Weitzenb\"ock identity \ref{verycrucialWB} in the last equation of \ref{derivativesoftildev}. Continuing further,
\begin{align}
&\frac{dJ_s}{ds} =\frac{2J_s t_0 f \vert \phi \vert_{\h_s}^2 (1-t_0 \vert \phi \vert_{\h_s}^2)}{(I)_s (II)_s}+\frac{it_0}{(I)_s(II)_s} \left(\partial \bar{\partial} (-f\vert \phi \vert_{\h_s}^2)+\partial \bar{\partial} f \vert \phi \vert_{\h_s}^2 - \Theta_{\h_s} f \vert \phi \vert_{\h_s}^2 \right) \nonumber \\
&= \frac{2J_s t_0 f \vert \phi \vert_{\h_s}^2 (1-t_0 \vert \phi \vert_{\h_s}^2)}{(I)_s (II)_s}+\frac{it_0}{(I)_s(II)_s} \left(-f \nabla_{\h_s} ^{1,0} \phi \wedge \nabla^{0,1}\phi^{\dag_{\h_s}} -\partial f \wedge \bar{\partial} \vert \phi \vert_{\h_s}^2 - \partial \vert \phi \vert_{\h_s}^2 \wedge \bar{\partial} f\right).
\label{derivjssecond}
\end{align}
Substituting \ref{derivjssecond} into \ref{FTC} we get the following equation.
\begin{align}
i\partial \bar{\partial}f &=\displaystyle \int_0 ^1 ds \Bigg ( f\vert\phi\vert_{\h_s}^2 J_s + (1-\vert \phi \vert_{\h_s}^2)  \frac{2J_s t_0 f \vert \phi \vert_{\h_s}^2 (1-t_0 \vert \phi \vert_{\h_s}^2)}{(I)_s (II)_s} \nonumber \\
&+ \frac{it_0(1-\vert \phi \vert_{\h_s}^2)}{(I)_s(II)_s} \left(-f \nabla_{\h_s} ^{1,0} \phi \wedge \nabla^{0,1}\phi^{\dag_{\h_s}} -\partial f \wedge \bar{\partial} \vert \phi \vert_{\h_s}^2 - \partial \vert \phi \vert_{\h_s}^2 \wedge \bar{\partial} f\right) \Bigg ).
\label{FTCsecond}
\end{align}
Before we proceed further, we note that the proof of Lemma \ref{lowerbound} implies the following lower bound.
\begin{gather}
f \geq -C_1,
\label{loweronf}
\end{gather}
where $C_1$ depends only on $r_2,\mu, h_0$. Define
\begin{gather}
\tilde{f} = f(\beta + \vert \phi \vert_{\h}^2),
\label{deftildef}
\end{gather}
where $\beta>1$ is a large enough constant (depending only on $r_2, \mu, h_0$) to be chosen later on and $\h$ is defined as
\begin{gather}
\h = \displaystyle \int_0 ^1 \h_s ds = \h_1 \int_0 ^1 e^{-sf}ds.
\label{definitionoffrakh}
\end{gather}
From \ref{boundonphi} it easily follows that 
\begin{gather}
\vert \phi \vert_{\h}^2 \leq 1.
\label{boundonphifrakh}
\end{gather}
Moreover, the curvature of $\h$ is as follows.
\begin{align}
\Theta_{\h} &= \Theta_{\h_1} - \partial \bar{\partial} \ln \left ( \displaystyle \int_0^1 e^{-sf}ds \right ) \nonumber \\
&= \Theta_{\h_1}+\partial\frac{\bar{\partial} f \displaystyle \int_0^1 se^{-sf}ds}{\displaystyle \int_0^1 e^{-sf}ds} \nonumber \\
&= \Theta_{\h_1}+\partial \bar{\partial} f \frac{\displaystyle \int_0^1 se^{-sf}ds}{\displaystyle \int_0^1 e^{-sf}ds}+\partial f \wedge \bar{\partial} f \left (\frac{\left ( \displaystyle \int_0^1 se^{-sf}ds \right )^2}{\left (\displaystyle \int_0^1 e^{-sf}ds \right)^2}-\frac{\displaystyle \int_0^1 s^2e^{-sf}ds}{\displaystyle \int_0^1 e^{-sf}ds} \right ).
\label{curvatureoffrakhinter}
\end{align}
By the Cauchy-Schwarz inequality we see that
\begin{align}
\frac{\left ( \displaystyle \int_0^1 se^{-sf}ds \right )^2}{\left (\displaystyle \int_0^1 e^{-sf}ds \right)^2}-\frac{\displaystyle \int_0^1 s^2e^{-sf}ds}{\displaystyle \int_0^1 e^{-sf}ds}  &\leq 0.
\label{CSinterfrakh}
\end{align}
Therefore,
\begin{align}
i\Theta_h &\leq \avs i\Theta_{\h_2} + (1-\avs)i\Theta_{\h_1},
\label{curvatureoffrakhfinal}
\end{align}
where
\begin{gather}
0< \avs = \frac{\displaystyle \int_0^1 se^{-sf}ds}{\displaystyle \int_0^1 e^{-sf}ds} \leq 1.
\label{definitionofavs}
\end{gather}
\indent If the maximum of $\tilde{f}$ occurs at $p$, then $\partial \tilde{f} (p)=\bar{\partial} \tilde{f}(p)=0$. Assumption \ref{assumption:onf} implies that $\tilde{f}(p)>0, f(p)>0$. Therefore,
\begin{align}
\partial f(p) (\beta+\vert \phi \vert_{\h}^2)(p)&=-f(p)\partial \vert \phi \vert_{\h}^2(p) \nonumber \\
\Rightarrow \partial f (p) &= \frac{-f(p)\partial (\vert \phi \vert_{\h}^2)(p)}{\beta+\vert \phi \vert_{\h}^2(p)}.
\label{firstderitildefuni}
\end{align}
Moreover, $i\partial \bar{\partial} \tilde{f} (p) \leq 0$, i.e.,
\begin{align}
0&\geq (\beta+\vert \phi \vert_{\h}^2(p)) i\partial \bar{\partial} f(p) + i\partial f (p) \wedge \bar{\partial} \vert \phi \vert_{\h}^2 (p) + i\partial \vert \phi \vert_{\h}^2 (p) \wedge \bar{\partial} f(p) \nonumber \\
&+ f(p)(-i\Theta_{\h}(p) \vert \phi \vert_{\h}^2(p)+ i\nabla_{\h}^{1,0} \phi (p) \wedge \nabla^{0,1} \phi^{\dag_{\h}}(p)) \nonumber \\
\Rightarrow 0 &\geq (\beta+\vert \phi \vert_{\h}^2(p)) i\partial \bar{\partial} f(p) - \frac{2i f (p) \vert \phi \vert_{\h}^2(p) \nabla^{1,0}_{\h} \phi (p)\wedge \nabla^{0,1} \phi^{\dag_{\h}} (p)}{\beta+\vert \phi \vert_{\h}^2(p)}\nonumber \\
&+ f(p)(-i\Theta_{\h} \vert \phi \vert_{\h}^2(p)+ i\nabla_{\h}^{1,0} \phi (p) \wedge \nabla^{0,1} \phi^{\dag_{\h}}(p)) \nonumber \\
\Rightarrow 0 &\geq i\partial \bar{\partial}f(p) -\frac{f(p)i\Theta_h(p) \vert \phi \vert_{\h}^2(p)}{\beta+\vert \phi \vert_{\h}^2(p)} + if(p)\nabla_{\h}^{1,0} \phi (p) \wedge \nabla^{0,1} \phi^{\dag_{\h}}(p) \frac{\beta- \vert \phi \vert_{\h}^2(p)}{(\beta+\vert \phi \vert_{\h}^2(p))^2} \nonumber \\
&\geq  i\partial \bar{\partial}f(p) -\frac{f(p)(\avs i\Theta_{\h_2} + (1-\avs)i\Theta_{\h_1})(p) \vert \phi \vert_{\h}^2(p)}{\beta+\vert \phi \vert_{\h}^2(p)}  \nonumber \\
&+ if(p)\nabla_{\h}^{1,0} \phi (p) \wedge \nabla^{0,1} \phi^{\dag_{\h}}(p) \frac{\beta- \vert \phi \vert_{\h}^2(p)}{(\beta+\vert \phi \vert_{\h}^2(p))^2},
\label{secondderitildefuniinter}
\end{align}
where we used \ref{curvatureoffrakhfinal} in the last inequality. Expanding further, we get the following.
\begin{align}
0 &\geq i\partial \bar{\partial}f(p) -\frac{f(p)(\avs i\Theta_{\h_2} + (1-\avs)i\Theta_{\h_1})(p) \vert \phi \vert_{\h}^2(p)}{\beta+\vert \phi \vert_{\h}^2(p)}  \nonumber \\
&+ if(p)\nabla_{\h}^{1,0} \phi (p) \wedge \nabla^{0,1} \phi^{\dag_{\h}}(p) \frac{\beta- \vert \phi \vert_{\h}^2(p)}{(\beta+\vert \phi \vert_{\h}^2(p))^2} \nonumber \\
\Rightarrow 0 &\geq i\partial \bar{\partial}f(p) -\frac{f(p)\left(\avs \frac{1-\vert \phi \vert_{\h_2}^2}{(I)_2(II)_2} + (1-\avs)\frac{1-\vert \phi \vert_{\h_1}^2}{(I)_1(II)_1}\right)(p) \vert \phi \vert_{\h}^2(p)\mu u^{1-t_0}(p)\omega_{\Sigma}(p)}{\beta+\vert \phi \vert_{\h}^2(p)}  \nonumber \\
&+ \frac{if(p)}{\beta+\vert \phi \vert_{\h}^2(p)} \Bigg ( \frac{\nabla_{\h}^{1,0} \phi (p) \wedge \nabla^{0,1} \phi^{\dag_{\h}}(p)(\beta- \vert \phi \vert_{\h}^2(p))}{\beta+\vert \phi \vert_{\h}^2(p)}\nonumber\\  &-\left(\avs \frac{(1-\vert \phi \vert_{\h_2}^2)\nabla_2^{1,0} \phi \wedge \nabla^{0,1} \phi^{\h_2}}{(I)_2(II)_2} + (1-\avs)\frac{(1-\vert \phi \vert_{\h_1}^2)\nabla_1^{1,0} \phi \wedge \nabla^{0,1} \phi^{\h_1}}{(I)_1(II)_1}\right)(p) t_0 \vert \phi \vert_{\h}^2(p) \Bigg ) .
\label{secondderitildefuni}
\end{align}
\indent Let 
$$E=\frac{\frac{\partial \vert \phi \vert_{\h}^2\bar{\partial} \vert \phi \vert_{\h_s}^2}{\beta+\vert \phi \vert_{\h}^2}+\frac{\partial \vert \phi \vert_{\h_s}^2\bar{\partial} \vert \phi \vert_{\h}^2}{\beta+\vert \phi \vert_{\h}^2}}{\nabla_{\h_s}^{1,0} \phi\wedge \nabla_{\h_s}^{0,1} \phi^{\dag_{\h_s}}}.$$
 Before substituting \ref{FTCsecond} in \ref{secondderitildefuni} we evaluate it at $p$ and simplify it using \ref{firstderitildefuni}. 
\begin{align}
i\partial \bar{\partial}f(p) &\geq \displaystyle \int_0 ^1 ds \Bigg ( f \vert \phi \vert_{\h_s}^2 (p) J_s(p) \nonumber \\
&+  \frac{it_0(1-\vert \phi \vert_{\h_s}^2)(p)}{(I)_s(p)(II)_s(p)} f(p) \nabla_{\h_s} ^{1,0} \phi (p) \wedge \nabla^{0,1}\phi^{\dag_{\h_s}}(p)\left(-1 +E  \right )\Bigg ) \nonumber \\
&\geq  \displaystyle \int_0 ^1 ds \Bigg ( f \vert \phi \vert_{\h_s}^2 (p)\omega_{\Sigma}(p) \frac{\mu u^{1-t_0}(p)}{(I)_s(p)(II)_s(p)} \nonumber \\
&+  \frac{it_0f(p)\nabla_{\h_s} ^{1,0} \phi (p) \wedge \nabla^{0,1}\phi^{\dag_{\h_s}}(p)}{(I)_s(p) (II)_s(p)}\left(-(1-\vert \phi \vert_{\h_s}^2)(p) +\vert \phi \vert_{\h_s}^2(p) + (1-\vert \phi \vert_{\h_s}^2)E\right )\Bigg )
\label{FTCsimplified}
\end{align}
Now we substitute \ref{FTCsimplified} in \ref{secondderitildefuni} to get the following equation. 
\begin{align}
0&\geq \displaystyle \int_0 ^1 ds \Bigg ( f \vert \phi \vert_{\h_s}^2 (p)\omega_{\Sigma}(p) \frac{\mu u^{1-t_0}(p)}{(I)_s(p)(II)_s(p)} \nonumber \\
&-  \frac{it_0f(p)\nabla_{\h_s} ^{1,0} \phi (p) \wedge \nabla^{0,1}\phi^{\dag_{\h_s}}(p)}{(I)_s(p) (II)_s(p)}\left(1+\vert E\vert \right )\Bigg ) \nonumber \\
&- \frac{f(p)\omega_{\Sigma}(p)\left(\avs \frac{1-\vert \phi \vert_{\h_2}^2}{(I)_2(II)_2} + (1-\avs)\frac{1-\vert \phi \vert_{\h_1}^2}{(I)_1(II)_1}\right)(p) \vert \phi \vert_{\h}^2(p)\mu u^{1-t_0}(p)}{\beta+\vert \phi \vert_{\h}^2(p)}  \nonumber \\
&+ \frac{if(p)}{\beta+\vert \phi \vert_{\h}^2(p)} \Bigg ( \frac{\nabla_{\h}^{1,0} \phi (p) \wedge \nabla^{0,1} \phi^{\dag_{\h}}(p)(\beta- \vert \phi \vert_{\h}^2(p))}{\beta+\vert \phi \vert_{\h}^2(p)}\nonumber\\  &-\left(\avs \frac{(1-\vert \phi \vert_{\h_2}^2)\nabla_2^{1,0} \phi \wedge \nabla^{0,1} \phi^{\h_2}}{(I)_2(II)_2} + (1-\avs)\frac{(1-\vert \phi \vert_{\h_1}^2)\nabla_1^{1,0} \phi \wedge \nabla^{0,1} \phi^{\h_1}}{(I)_1(II)_1}\right)(p) \vert \phi \vert_{\h}^2(p) \Bigg ).
\label{aftersubstitution}
\end{align}
We simplify \ref{aftersubstitution} further to get the following inequality at $p$. (We suppress the dependence on $p$ from now onwards.)
\begin{align}
0&\geq \displaystyle \int_0 ^1 ds \Bigg (\vert \phi \vert_{\h_s}^2 f\omega_{\Sigma}\mu u^{1-t_0} \left [ \frac{1 }{(I)_s(II)_s} - \frac{  \avs \frac{1-\vert \phi \vert_{\h_2}^2}{(I)_2(II)_2} + (1-\avs)\frac{1-\vert \phi \vert_{\h_1}^2}{(I)_1(II)_1}}{\beta+\vert \phi \vert_{\h}^2} \right]  \nonumber \\
&-  \frac{it_0f\nabla_{\h_s} ^{1,0} \phi  \wedge \nabla^{0,1}\phi^{\dag_{\h_s}}}{(I)_s (II)_s}\left(1+\vert E \vert\right )\Bigg ) \nonumber \\
&+ \frac{if}{\beta+\vert \phi \vert_{\h}^2} \Bigg ( \frac{\nabla_{\h}^{1,0} \phi  \wedge \nabla^{0,1} \phi^{\dag_{\h}}(\beta- \vert \phi \vert_{\h}^2)}{\beta+\vert \phi \vert_{\h}^2}\nonumber\\  
&-\left(\avs \frac{(1-\vert \phi \vert_{\h_2}^2)\nabla_2^{1,0} \phi \wedge \nabla^{0,1} \phi^{\h_2}}{(I)_2(II)_2} + (1-\avs)\frac{(1-\vert \phi \vert_{\h_1}^2)\nabla_1^{1,0} \phi \wedge \nabla^{0,1} \phi^{\h_1}}{(I)_1(II)_1}\right) \vert \phi \vert_{\h}^2 \Bigg ).
\label{simplifyaftersubstitution}
\end{align}
\indent At this point we choose $\beta>1$ to be so large that 
\begin{align}
&\frac{1}{(I)_s(II)_s} - \frac{\avs \frac{1-\vert \phi \vert_{\h_2}^2}{(I)_2(II)_2} + (1-\avs)\frac{1-\vert \phi \vert_{\h_1}^2}{(I)_1(II)_1}}{\beta+\vert \phi \vert_{\h}^2} \nonumber \\
&= \avs \left ( \frac{1}{(I)_s(II)_s} - \frac{1-\vert \phi \vert_{\h_2}^2}{(I)_2(II)_2 (\beta+\vert \phi \vert_{\h}^2)} \right) +(1-\avs) \left (\frac{1}{(I)_s(II)_s}-\frac{1-\vert \phi \vert_{\h_1}^2}{(I)_1(II)_1 (\beta+\vert \phi \vert_{\h}^2)} \right ) \nonumber \\
&> \frac{1}{10(2r_2)(2r_2+1)}.
\label{choiceofbeta}
\end{align}
\indent Using \ref{choiceofbeta} in \ref{simplifyaftersubstitution} we get the following inequality at $p$.
\begin{align}
0&\geq \displaystyle \int_0 ^1 ds \Bigg ( \frac{f\vert \phi \vert_{\h_s}^2 \mu u^{1-t_0}}{10(2r_2)(2r_2+1)} -  \frac{it_0f\nabla_{\h_s} ^{1,0} \phi  \wedge \nabla^{0,1}\phi^{\dag_{\h_s}}}{(I)_s (II)_s}\left(1+\vert E \vert \right )\Bigg ) \nonumber \\
&+ \frac{if}{\beta+\vert \phi \vert_{\h}^2} \Bigg ( \frac{\nabla_{\h}^{1,0} \phi  \wedge \nabla^{0,1} \phi^{\dag_{\h}}(\beta- \vert \phi \vert_{\h}^2)}{\beta+\vert \phi \vert_{\h}^2}\nonumber\\  
&-\left(\avs \frac{(1-\vert \phi \vert_{\h_2}^2)\nabla_2^{1,0} \phi \wedge \nabla^{0,1} \phi^{\h_2}}{(I)_2(II)_2} + (1-\avs)\frac{(1-\vert \phi \vert_{\h_1}^2)\nabla_1^{1,0} \phi \wedge \nabla^{0,1} \phi^{\h_1}}{(I)_1(II)_1}\right) \vert \phi \vert_{\h}^2 \Bigg ).
\label{almostthereuniq}
\end{align}
Now we relate $\nabla_{\h_1}\phi$  to $\nabla_{\h}\phi$ using \ref{firstderitildefuni}.
\begin{align}
\nabla_{\h_1}^{1,0} \phi \wedge \nabla^{0,1} \phi^{\dag_{\h_1}} &=\left (\nabla_{\h}^{1,0}\phi -\partial \ln \left (\displaystyle \int_0^1 e^{-sf}ds \right )\phi \right )\wedge  \nabla^{0,1} \left ( \phi^{\dag_{\h}} \left(\displaystyle \int_0^1 e^{-sf} ds \right)^{-1}\right )\nonumber \\
&=  \frac{\left (\nabla_{\h}^{1,0}\phi +\avs \partial f \phi \right )\wedge \left (\nabla^{0,1}\phi^{\dag_{\h}} +\avs \bar{\partial} f \phi^{\dag_{h}} \right )}{\displaystyle \int_0^1 e^{-sf} ds}\nonumber \\
&= \frac{\left (\nabla_{\h}^{1,0}\phi -\avs \frac{f\phi\partial \vert \phi \vert_{\h}^2}{\beta +\vert \phi \vert_{\h}^2} \right )\wedge \left (\nabla^{0,1}\phi^{\dag_{\h}} -\avs  \frac{f\phi^{\dag_{h}}\bar{\partial}\vert \phi \vert_{\h}^2  }{\beta +\vert \phi \vert_{\h}^2} \right )}{\displaystyle \int_0^1 e^{-sf} ds}\nonumber \\
&=\frac{ \nabla_{\h}^{1,0}\phi \wedge \nabla^{0,1}\phi^{\dag_{\h}}}{\displaystyle \int_0^1 e^{-sf} ds} \left (1- \frac{\avs f \vert \phi \vert_{\h}^2}{\beta+\vert \phi \vert_{\h}^2} \right ) ^2 \nonumber \\
&= \frac{\nabla_{\h}^{1,0}\phi \wedge \nabla^{0,1}\phi^{\dag_{\h}}}{\displaystyle \int_0^1 e^{-sf} ds} \left (1+ \frac{ \int_0^1 (-tf)e^{-tf}dt \vert \phi \vert_{\h_1}^2}{\beta+\vert \phi \vert_{\h}^2} \right ) ^2.
\label{nablarelations1}
\end{align}
Note that whenever $x<C$, $xe^x$ is bounded above and below. Therefore, since $-tf \leq tC_1$ (where $C_1$ is the constant appearing in \ref{loweronf}), 
\begin{gather}
\displaystyle \vert \int_0^1 (-tf)e^{-tf}dt \vert \leq K,
\end{gather}
where $K$ depends only on $r_2, \mu, h_0$. This means that 
\begin{align}
\vert \phi \vert_{\h}^2 \nabla_{\h_1}^{1,0} \phi \wedge \nabla^{0,1} \phi^{\dag_{\h_1}} &\leq \vert \phi \vert_{\h} ^2 \frac{\nabla_{\h}^{1,0}\phi \wedge \nabla^{0,1}\phi^{\dag_{\h}}}{\displaystyle \int_0 ^1 e^{-sf}ds} \left (1+ \frac{K}{\beta} \right ) ^2 \nonumber \\
&\leq  \vert \phi \vert_{\h_1}^2\nabla_{\h}^{1,0}\phi \wedge \nabla^{0,1}\phi^{\dag_{\h}} \left (1+ \frac{K}{\beta} \right ) ^2 \nonumber \\
&\leq  \nabla_{\h}^{1,0}\phi \wedge \nabla^{0,1}\phi^{\dag_{\h}} \left (1+ \frac{K}{\beta} \right ) ^2.
\label{asmallinequality}
\end{align} 
Likewise, for $\nabla_{\h_2} \phi$ we have the following relation.
\begin{align}
\nabla_{\h_2}^{1,0} \phi \wedge \nabla^{0,1} \phi^{\dag_{\h_2}} &=\left (\nabla_{\h}^{1,0}\phi -\partial \ln \left (\displaystyle \int_0^1 e^{-sf}ds \right )\phi -\partial f \phi \right )\wedge  \nabla^{0,1} \left ( \phi^{\dag_{\h}} e^{-f}\left(\displaystyle \int_0^1 e^{-sf} ds \right)^{-1}\right )\nonumber \\
&=  e^{-f}\frac{\left (\nabla_{\h}^{1,0}\phi +(\avs-1) \partial f \phi \right )\wedge \left (\nabla^{0,1}\phi^{\dag_{\h}} +(\avs-1) \bar{\partial} f \phi^{\dag_{h}} \right )}{\displaystyle \int_0^1 e^{-sf} ds}\nonumber \\
&= e^{-f}\frac{\left (\nabla_{\h}^{1,0}\phi -(\avs-1) \frac{f\phi\partial \vert \phi \vert_{\h}^2}{\beta +\vert \phi \vert_{\h}^2} \right )\wedge \left (\nabla^{0,1}\phi^{\dag_{\h}} -(\avs-1)  \frac{f\phi^{\dag_{h}}\bar{\partial}\vert \phi \vert_{\h}^2  }{\beta +\vert \phi \vert_{\h}^2} \right )}{\displaystyle \int_0^1 e^{-sf} ds}\nonumber \\
&=e^{-f}\frac{ \nabla_{\h}^{1,0}\phi \wedge \nabla^{0,1}\phi^{\dag_{\h}}}{\displaystyle \int_0^1 e^{-sf} ds} \left (1- \frac{(\avs-1) f \vert \phi \vert_{\h}^2}{\beta+\vert \phi \vert_{\h}^2} \right ) ^2 \nonumber \\
&= e^{-f}\frac{\nabla_{\h}^{1,0}\phi \wedge \nabla^{0,1}\phi^{\dag_{\h}}}{\displaystyle \int_0^1 e^{-sf} ds} \left (1+ \frac{ \int_0^1 (1-t)fe^{(1-t)f}dt \vert \phi \vert_{\h_2}^2}{\beta+\vert \phi \vert_{\h}^2} \right ) ^2.
\label{nablarelations2}
\end{align}
By the same reasoning as before, 
\begin{align}
\vert \phi \vert_{\h}^2 \nabla_{\h_2}^{1,0} \phi \wedge \nabla^{0,1} \phi^{\dag_{\h_2}} &\leq \vert \phi \vert_{\h} ^2 e^{-f}\frac{\nabla_{\h}^{1,0}\phi \wedge \nabla^{0,1}\phi^{\dag_{\h}}}{\displaystyle \int_0 ^1 e^{-sf}ds} \left (1+ \frac{K}{\beta} \right ) ^2 \nonumber \\
&\leq  \vert \phi \vert_{\h_2}^2\nabla_{\h}^{1,0}\phi \wedge \nabla^{0,1}\phi^{\dag_{\h}} \left (1+ \frac{K}{\beta} \right ) ^2 \nonumber \\
&\leq  \nabla_{\h}^{1,0}\phi \wedge \nabla^{0,1}\phi^{\dag_{\h}} \left (1+ \frac{K}{\beta} \right ) ^2.
\label{bsmallinequality}
\end{align}
Putting \ref{asmallinequality}, \ref{bsmallinequality}, and \ref{almostthereuniq} we get the following inequality at $p$.
\begin{align}
0&\geq \displaystyle \int_0 ^1 ds \Bigg ( \frac{f\vert \phi \vert_{\h_s}^2 \mu u^{1-t_0}}{10(2r_2)(2r_2+1)} -  \frac{it_0f\nabla_{\h_s} ^{1,0} \phi  \wedge \nabla^{0,1}\phi^{\dag_{\h_s}}}{(I)_s (II)_s}\left(1+\vert E \vert \right )\Bigg ) \nonumber \\
&+ \frac{if \nabla_{\h}^{1,0} \phi  \wedge \nabla^{0,1} \phi^{\dag_{\h}}}{\beta+\vert \phi \vert_{\h}^2} \Bigg ( \frac{\beta- \vert \phi \vert_{\h}^2}{\beta+\vert \phi \vert_{\h}^2}-\left(1+\frac{K}{\beta} \right)^2\left( \frac{\avs}{(I)_2(II)_2} + \frac{(1-\avs)}{(I)_1(II)_1}\right) \Bigg ).
\label{furtherthaneverbeforeuniq}
\end{align}
For a large enough $\beta$ (depending only on $r_2, \mu, h_0$) we see that
\begin{align}
0&\geq \displaystyle \int_0 ^1 ds \Bigg ( \frac{f\vert \phi \vert_{\h_s}^2 \mu u^{1-t_0}}{10(2r_2)(2r_2+1)} -  t_0\frac{if\nabla_{\h_s} ^{1,0} \phi  \wedge \nabla^{0,1}\phi^{\dag_{\h_s}}(1+\vert E \vert)}{2r_2(2r_2+1)}  \Bigg )\nonumber\\
&+ \frac{if \nabla_{\h}^{1,0} \phi  \wedge \nabla^{0,1} \phi^{\dag_{\h}}}{\beta+\vert \phi \vert_{\h}^2} \left (1-\frac{3}{2r_2(2r_2+1)} \right ) \nonumber \\
&\geq \displaystyle \int_0 ^1 ds \Bigg ( \frac{f\vert \phi \vert_{\h_s}^2 \mu u^{1-t_0}}{10(2r_2)(2r_2+1)} -  t_0\frac{if\nabla_{\h_s} ^{1,0} \phi  \wedge \nabla^{0,1}\phi^{\dag_{\h_s}}(1+\vert E \vert)}{2r_2(2r_2+1)}  \Bigg )+ \frac{if \nabla_{\h}^{1,0} \phi  \wedge \nabla^{0,1} \phi^{\dag_{\h}}}{2\beta}.
\label{evenfurtherthaneverbeforeuniq}
\end{align}
The following equations describe the relationship between $\nabla_{\h} \phi$ and $\nabla_{\h_s} \phi$.
\begin{align}
\nabla_{\h_s} ^{1,0} \phi \wedge \nabla ^{0,1} \phi^{\dag_{\h_s}} &= \left(\nabla_{\h}^{1,0} \phi -\partial \ln \left ( \displaystyle \int_0^1 e^{-tf}dt\right) -s\partial f \phi  \right)\wedge \nabla^{0,1} \left( \phi^{\dag}_{\h} e^{-sf} \left (\displaystyle \int_0 ^1 e^{-tf}dt \right) ^{-1}\right)\nonumber \\
&=\frac{e^{-sf}}{\displaystyle \int_0 ^1 e^{-tf}dt}\left(\nabla_{\h}^{1,0} \phi +(\avs-s)\partial f \phi  \right)\wedge \left(\nabla^{0,1} \phi^{\dag_{\h}} +(\avs-s)\bar{\partial} f \phi^{\dag_{\h}}  \right)\nonumber \\
&=\frac{e^{-sf}}{\displaystyle \int_0 ^1 e^{-tf}dt}\nabla_{\h}^{1,0} \phi\wedge \nabla^{0,1} \phi^{\dag_{\h}} \left (1+\frac{(s-\avs)f \vert \phi \vert_{\h}^2}{\beta+\vert \phi \vert_{\h}^2} \right )^2
\label{relationhhs}
\end{align}
Note that
\begin{align}
\left \vert \frac{(s-\avs)f \vert \phi \vert_{\h}^2}{\beta+\vert \phi \vert_{\h}^2} \right \vert &\leq \frac{2}{\beta} \vert f \vert \vert \phi \vert_{\h_1}^2 \displaystyle \int_0^1 e^{-tf} dt \nonumber \\
&\leq \frac{2}{\beta} \vert 1-e^{-f}\vert  \nonumber \\
&\leq \frac{2(1+e^{C_1})}{\beta},
\label{upperonaterm}
\end{align}
where we used estimate \ref{loweronf}. Using \ref{evenfurtherthaneverbeforeuniq}, \ref{relationhhs}, and \ref{upperonaterm} we see that
\begin{align}
0&\geq \frac{f\vert \phi \vert_{\h}^2 \mu u^{1-t_0}}{10(2r_2)(2r_2+1)}+if  \nabla_{\h}^{1,0} \phi  \wedge \nabla^{0,1} \phi^{\dag_{\h}}\left (\frac{1}{2\beta}-\frac{t_0\left(1+\displaystyle \int_0^1\vert E \vert \frac{e^{-sf}}{\int_0^1 e^{-tf} dt}ds\right)}{2r_2(2r_2+1)} \left (1+ \frac{2(1+e^{C_1})}{\beta}\right)^2 \right ).
\label{almostsemifinalunique}
\end{align}
Now we estimate $\int_0^1 \vert E \vert \frac{e^{-sf}}{\int_0^1 e^{-tf} dt} ds$. 
\begin{gather}
\displaystyle \int_0^1 \vert E \vert \frac{e^{-sf}}{\int_0^1 e^{-tf} dt} ds \leq \frac{2}{\beta+\vert \phi \vert_{\h}^2} \int_0^1 \frac{e^{-sf}}{\int_0^1 e^{-tf} dt} \vert \phi \vert_{\h} \vert \phi \vert_{\h_s} \frac{\vert \nabla^{1,0}_{\h} \phi \vert_{\h}}{\vert \nabla ^{1,0}_{\h_s} \phi \vert_{\h_s}}  ds \nonumber \\
\leq \frac{2}{\beta}\int_0^1 \frac{e^{-sf}}{\int_0^1 e^{-tf} dt} \frac{\vert \nabla^{1,0}_{\h} \phi \vert}{\vert \nabla^{1,0}_{\h}\phi+\phi\partial f (s-\langle s\rangle) \vert_{\h} \frac{\sqrt{e^{-sf}}}{\sqrt{\int_0^1 e^{-tf} dt}}} ds.
\label{estimateforerrorterm}
\end{gather}
Using \ref{firstderitildefuni} we get the following.
\begin{gather}
\displaystyle \int_0^1 \vert E \vert \frac{e^{-sf}}{\int_0^1 e^{-tf} dt} ds \leq \frac{2}{\beta}\int_0^1 \frac{e^{-sf}}{\int_0^1 e^{-tf} dt} \frac{1}{\vert 1-\frac{\vert \phi \vert_{\h}^2 \vert s-\langle s \rangle \vert}{\beta+\vert \phi \vert_{\h}^2} \vert \frac{\sqrt{e^{-sf}}}{\sqrt{\int_0^1 e^{-tf} dt}}}ds \nonumber \\
\leq 1,
\label{almostthereestimateforerrorterm}
\end{gather}
for large enough $\beta$. Therefore,
\begin{align}
0&\geq \frac{f\vert \phi \vert_{\h}^2 \mu u^{1-t_0}}{10(2r_2)(2r_2+1)}+\frac{if  \nabla_{\h}^{1,0} \phi  \wedge \nabla^{0,1} \phi^{\dag_{\h}}}{4\beta}
\label{almostfinalunique}
\end{align}
for sufficiently small $t_0$ (depending on $\beta, r_2, \mu, h_0$). As before, since the line bundle is of degree $1$, either $\phi(p)\neq 0$ or $\vert \nabla \phi \vert (p) \neq 0$. This implies that $f(p)\leq 0$ which is a contradiction. Hence $f\equiv 0$ showing uniqueness for small $t_0$.
\end{proof}

\indent Now we complete step 4 and hence the proof of uniqueness. 
\begin{lemma}
If there exists a $t_0 \in [0,1]$ such that $h_{t_0}=\tilde{h}_{t_0}$ then $h_1=h_2$.
\end{lemma}
\begin{proof}
Let $S \subset [\delta_0,1]$ be the set of all $t$ such that $h_t=\tilde{h}_t$. Then $S$ satisfies the following.
\begin{enumerate}
\item It is non-empty : $\delta_0 \in S$.
\item It is open : The proof of openness (Lemma \ref{openness}) and the inverse function theorem of Banach manifolds shows that locally the solution is unique and hence $S$ is open.
\item It is closed : The \emph{a priori} estimates in the proof of existence show that $S$ is closed.
\end{enumerate}
Therefore $S=[\delta_0,1]$ as desired.
\end{proof}

\end{document}